\documentclass{article}  % list options between brackets
\usepackage{amssymb}  % list packages between braces
\usepackage{amsthm}
\usepackage{eucal}
\usepackage{verbatim}
\usepackage[dvips]{graphicx}
\usepackage{multirow}
\usepackage{fancyhdr}
\usepackage{color}
\usepackage{enumerate}
\usepackage{calrsfs}
\usepackage{fullpage}
\usepackage{hyperref}
\usepackage{amsmath}
\DeclareMathAlphabet{\pazocal}{OMS}{zplm}{m}{n}
\DeclareMathAlphabet{\mathcalligra}{T1}{calligra}{m}{n}
\DeclareMathOperator*{\esssup}{ess\,sup}

\newtheorem{theorem}{Theorem}
\newtheorem{lemma}{Lemma}
\newtheorem{definition}{Definition}
\newtheorem{remark}{Remark}
\newtheorem{proposition}{Proposition}
\newtheorem{corollary}{Corollary}

\makeatletter
\newcommand{\leqnomode}{\tagsleft@true}
\newcommand{\reqnomode}{\tagsleft@false}
\newcommand*\diff{\mathop{}\!\mathrm{d}}
\makeatother
\def\({\begin{eqnarray}}
\def\){\end{eqnarray}}
\def\[{\begin{eqnarray*}}
\def\]{\end{eqnarray*}}
\def\part#1#2{\frac{\partial #1}{\partial #2}}

\title{Opinion dynamics for an increasing population of agents. A symmetric continuous agent model.}
\author{Ioannis Markou}

\begin{document}

\maketitle

\begin{abstract}
In this paper we formulate a continuous opinion model that takes into account population growth, i.e. increase with time in the number of interacting agents $N(t)$. In our setting the population growth is governed by a generic growth rate function $b(t, N(t))$. The two main components of our model are the growth rate $b(t, N(t))$, as well as the opinions of the incoming agents which are modeled in our system as boundary conditions in a free boundary problem. We give results on the well-posedness of the model and results that showcase how these two components affect the long time asymptotic behavior of our system. Moreover, we provide a kinetic (probabilistic) description of our model and give results on well-posedness and asymptotics for the kinetic model.
\end{abstract}



\section{Introduction}
\label{sec: Introduction}
The past few decades have witnessed an unprecedented soaring interest in the study of multi-agent Individual Based models (IBM). IBMs aim to explain the emergence of complex patterns that appear in natural and social phenomena starting from very simple local rules that govern the interactions between individuals. Of particular interest among IBMs are models of social dynamics which study the evolution of opinions among interacting agents with the aim of finding the possibility of consensus in opinions. The literature for the study of social dynamics models, including opinion models, is vast and for that reason we direct the reader to the excellent review article \cite{CaFoLo}. A specific instance of an opinion model is the Krause dynamical model introduced by U. Krause in \cite{Kra} (oftentimes referred as the Hegselmann-Krause model due to \cite{HegKraus}). In this model opinions are represented by real variables or more generally by real vectors in $\mathbb{R}^d$. After each time step every agent updates their opinion by averaging all the opinions that lie in some neighborhood, with the metric distance being the usual Euclidean distance. From its introduction, the research in the Krause model has taken several directions see e.g. \cite{AyDu, BlHeTs1, BlHeTs2, CaFaTi, Dut, Has1, JabMot, NuGoWo, PiDuTr, Tos} etc.
The focus of this paper is to study the effect of population growth in a simple model of opinion dynamics, i.e. a symmetric version of the Krause model that is also continuous in the space of agents.

Population growth, and more generally population change in IBM systems is not just a matter of academic curiosity but a phenomenon that should be taken into account in many social and physical systems. The results of population growth in systems where individuals interact by exchanging opinions, trading goods, compete with each other etc. are all too important to be ignored. Below, we give some concrete examples of systems where collective behavior is studied and it is of utmost importance to take into consideration population increase or change with time:
\\ (i) \textit{Expanding communities}: Population growth has many effects in the demographics of communities of all scales, ranging from small groups that share some common trait, all the way up to large nations. At the same time, even groups of people that do not necessarily share the same ethnic identity (such as religious groups, groups that belong in certain socio-economic class etc.) grow in power as their underlying population expands. Around the time people reach adulthood they obtain voting rights in most countries and they can take part in elections and referendums. A constantly growing population means that the groups of people that take part in the decision making process expand as well. It is therefore an erroneous idea to assume that systems in which people interact exchanging opinions and reaching consensus are always of a fixed population.
\\ (ii) \textit{Growing economies}: The effects of population growth to the economies of nations is a reality too important to be ignored. All countries strive to achieve economic expansion, which to a large extend is based on their population growth. One particular instance of how population expansion affects economy is in the redistribution of wealth among the population. New individuals constantly enter the economy with a specific consumer power participating in the exchange of goods and wealth production. One expects the effect of newcoming members in the economy to change the trading propensity of large populations resulting in possible changes in the wealth distribution curve (see e.g. \cite{CaMo, CoPaTo}).
\\ (iii) \textit{Online gaming}: In the past couple of decades the gaming industry has seen tremendous growth in the number of participants, since the internet helped with the globalization of this industry. Online gaming communities usually categorize players by strength by assigning a rating to each player. One very successful way to rate people who participate in gaming is the ELO rating system \cite{Elo} (see also \cite{DurEvWol, DurTorWol, JabJun}). According to this rating system, the strength of each player is achieved by their success rates with other players. In other words, the rating difference between two players measures the probability of outcomes in a game between them. The constant influx of new players in online playing sites has a direct effect on the distribution of rating among all players. Players that enter the system are assigned a provisional rating that does not correspond to their true rating strength. As a result, over time the ratings of players might be inflating or deflating depending on whether the provisional rating of incoming players is high or low compared to their true rating.

\textbf{Population growth equation.} In this paper we assume a growing population $N_t:=N(t)$ increasing with time, governed by the simple equation ($\dot{N}_t$ is the time derivative of $N_t$)
\begin{equation} \label{eq: growth} \dot{N}_t=b(t,N_t)N_t, \quad t \geq 0 ,\end{equation}
for a continuous, nonnegative growth rate $b(t,N_t)\geq 0$ and initial population $N_0$. Note that eq. \eqref{eq: growth} is another way of expressing a growth equation of the general form $\dot{N}_t=F(t, N_t)$ (for nonnegative $F(t, N_t) \geq 0$), but we choose to describe it as in \eqref{eq: growth} for convenience. Since $N_t$ is increasing in time it converges to some value $N_{\infty}:=\sup \limits_{t \geq 0} N_t$, which can be either finite ($<+\infty$), or infinite ($+\infty$). The semi-explicit solution of \eqref{eq: growth} is given by $N_t=e^{\int_0^t b(s,N_s) \diff s}N_0$, which implies that the condition for $N_{\infty}<\infty$ is $\int_0^{\infty} b(s,N_s) \diff s <\infty$ and likewise $N_\infty= \infty$ iff $\int_0^\infty b(s, N_s)\, \diff s=\infty$. Typically, the choice of growth rate $b(t, N_t)$ is of paramount importance to the long time asymptotics of our system. For instance, an integrable $b(t, N_t)$ implies that in the long term the population converges to some fixed value.
\begin{remark}
A specific instance is a rate function $b(t)$, in which case $N_\infty$ is given by the explicit formula $N_\infty= e^{\int_0^\infty b(s)\, \diff s}N_0$.
More specifically, we can consider that for $t \gg 1$, $b(t)$ behaves like $b(t) \sim t^{-\alpha}$ (for $\alpha \in \mathbb{R}$), where
$N_\infty =\infty$ iff $\alpha \leq 1$. The critical case $\alpha=1$ gives $N_t \sim t N_0$ which is practically  \textbf{linear} population growth. For $\alpha=0$ then $N_t \sim e^{t}N_0$ which corresponds to \textbf{exponential} growth. For the intermediate values $0 < \alpha <1$ then $N_t \sim e^{t^{1-\alpha }}N_0$. This corresponds to growth that is \textbf{subexponential} or superlinear (between linear and exponential). Finally, when $\alpha <0$ the growth is \textbf{superexponential}.
\end{remark}
\textbf{Opinions of incoming agents.} The next major assumption we make, other than the rule for the population growth, is that we have to prescribe some arbitrary ``opinion'' value to the incoming population. We should mention that the rule that decides the opinions of incoming agents is an obvious modeling assumption and we identify two classes of assumptions for the incoming population. The first class is a deterministic rule that assigns to each agent that is added at time $t$ a specific opinion $X(t, N_t)$. We make the conscious choice of considering a profile for the incoming opinions that is not just a function of time, but quite possibly of the population $N_t$, since the opinions of incoming agents at time $t$ might also depend on the growth rates $b(t, N_t)$. Another very realistic assumption for opinions of the incoming population is to consider that opinions incoming at time $t$ follow some stochastic process (with zero mean so that the average of incoming population does not drift to infinity!). In this paper we focus only on the assumption of deterministic incoming opinions. For  a simple model involving random incoming opinions see \cite{Mar2}.

It is very intuitive to assume that both the profile of the opinions for incoming agents and the growth rate of the population have a direct effect on the limiting distribution of opinions. In particular, the fastest the growth rate, the more one expects that the initial opinion profile plays little role on the final opinions, and on the other hand the opinion profile of incoming agents plays the major role. To be more precise, one expects that when $N_\infty=\infty$ then there should be no contribution of the initial opinions to the behavior of the aggregate of opinions at infinity, as it turns out being the case.

The remainder of this paper is organized as follows. In Section \ref{sec: Main results} we present our novel model and we give a characterization of the type of solutions that this model might admit for well-posedness. We also discuss briefly the macroscopic description of the model and present the main results. In Section \ref{sec: Well-Posedness} we give results on the well-posedness of the model. In Section \ref{sec: Moments} we study the moments of our model, and in Section \ref{sec: Convergence} we give several results on asymptotics with the help of a variance functional we introduce. In Section \ref{sec: Macroscopic} we give a kinetic formulation to our problem and present results on the asymptotics of the kinetic problem directly in relation with the continuous one. Finally, in Section \ref{sec: Conclusions} we draw conclusions and present many open problems that are related to our model but most importantly we give several interesting cases where our modeling idea can be implemented.

\section{Presentation of the model and main results.}
\label{sec: Main results}

The general $N$ agent consensus model in continuous time satisfies the system of equations
\begin{equation} \label{eq: consensus} \dot{x}_i =\sum \limits_{j=1}^N a_{ij} (x_j -x_i), \qquad 1 \leq i \leq N, \quad t \geq 0 . \end{equation} Each equation in \eqref{eq: consensus} describes the evolution in time of a certain state (or ``opinion'') variable $x_i(t) \in \mathbb{R}^d$, for $i=1, \ldots, N$. This system is supplemented with some initial conditions $x_i(0)=x_{i0}$. The most important component of the consensus model is the interaction matrix $\{a_{ij}\}_{i,j=1,\ldots,N}$. For our purposes the interaction coefficients $a_{ij}$ depend on the positions of the agents, and in particular of the metric distances between agents.

More specifically there are two variations of coefficient matrices that are mostly in use when the interactions depend on the metric distance between agents:
\begin{itemize}
\item (i) The symmetric case, where $a_{ij}=\frac{1}{N}\psi(|x_i -x_j|)$, for some non-increasing function $\psi(r)$ (possibly with compact support!)
\item (ii) The coefficients that were introduced in the original treatment of the consensus model by Hegselmann and Krause correspond to a stochastic matrix with elements $a_{ij}=\psi(|x_i -x_j|)/ \sum \limits_{k=1}^N \psi(|x_i-x_k|)$ (see \cite{HegKraus}).
 \end{itemize}

The influence functions $\psi(r)$ that are of theoretical and practical importance generally belong in two main groups. The first category assumes a non-increasing and Lipschitz continuous $\psi(r)$ with a
Lipschitz constant $L_{\psi}>0$, i.e.
\begin{equation*} \label{eq: Psi_Lipsch} |\psi(r_1)-\psi(r_2)|\leq L_{\psi} |r_1-r_2|, \quad \forall r_1,r_2 >0 .\end{equation*}
We mention that such a function may or may not have a compact support on $[0,\infty)$. We can choose for instance a continuous c.w. $\psi(r)$ compactly supported on $[0,1)$ which is Lipschitz on this interval with $\psi(1^-)=0$. Otherwise another choice might be a Lipschitz continuous, non-increasing $\psi(r)$ which is positive everywhere.

Another common choice of influence function $\psi(r)$ assumes a non-increasing function with support on some interval $[0, R_0]$ (we can always set $R_0=1$), with a jump discontinuity at $r=1$ i.e. a function described by
\begin{equation} \label{eq: Psi_Discont} \psi(r)= \left\{ \begin{array}{ll}  \psi_0(r)  \qquad  \text{for} \, \, r  \leq 1, \\ 0  \qquad \text{for} \, \,  r > 1 ,\end{array} \right.\end{equation}
with $\psi_0(1^-)\neq 0$.
The most common example used in the literature of the HK model is the step function with $\psi_0(r):=1$. Obviously, from a practical perspective the most important and classical choice of a communication weight is exactly this, but in this paper we will work with Lipschitz continuous communication weights. We understand that from a practical point of vies, the most classical examples are weights of type \eqref{eq: Psi_Discont}, but we want to focus on the difficulties introduced by the population growth effect, and not the ones relevant to working with a non regularity. Therefore, we treat two main instances of weights  \begin{itemize}  \item \textbf{Type I}: An everywhere positive $\psi(r)>0$ (for $r \geq 0$), not necessarily nonincreasing. The only assumption besides L. continuity is the global boundedness i.e. that $\psi_M:=\sup \limits_{r \geq 0} \psi(r)<\infty$ and \item \textbf{Type II}: a communication weight with a compact support on $[0,1]$ which is L. continuous everywhere (even at $r=1$).  \end{itemize}

In \cite{BlHeTs1, BlHeTs2} the authors proposed a opinion models in the space of continuous agents in order to better understand the consensus dynamics of discrete agents. Their model, in its symmetric setting satisfies the continuous integro-differential equation
\begin{equation*} \dot{x}_t(s)=\int_0^1 \psi(|x_t(s')-x_t(s)|) (x_t(s')-x_t(s))\, \diff s' , \quad 0 \leq s \leq 1 ,\end{equation*}
where $\psi(r)$ is the indicator of the unit interval $[0,1]$, i.e. $\psi(r):=\mathbf{1}_{[0,1]}(r)$.

The opinion model that we introduce taking into account the population growth carries a similar structure and is characterized by the following evolution for the state variable $x_t(s)\in \mathbb{R}^d$ for $t \geq 0$,
\begin{equation} \label{eq: conses_symm} \dot{x}_t(s)=\frac{1}{N_t} \int_0 ^{N_t} \psi(|x_t(s')-x_t(s)|) (x_t(s')-x_t(s))\, \diff s' , \quad 0 \leq s < N_t .\end{equation} Here the continuous variable $s$ takes values in the interval $[0,N_t)$ that is expanding
with time. As initial condition we prescribe an initial profile $x_0(s)$ defined on the interval $[0, N_0]$.

As a way of defining this problem properly we need to provide not just initial condition for $s \in [0,N_0)$ but also a condition that assigns values to the opinions of incoming agents at time $t$, i.e. define the value for agent $s=N_t$ at time $t$. This implies that we need to supplement our model with the condition for the free boundary that characterizes states of incoming agents, i.e.
\begin{equation} \label{eq: Mov_Bound_Cond} x_t(N_t)=X(t,N_t), \quad \text{for} \quad t \geq 0, \end{equation} where $X(t,N_t)$ is some function for $t \geq 0$. The condition on the free boundary is necessary for the well-posedeness of the continuous model. We will assume that the $X(t, N_t)$ function is bounded uniformly by some value $X_B$, i.e.  \begin{equation} \label{eq: X(t)_cond} |X(t, N_t)|\leq X_B \qquad \forall t \geq 0. \end{equation}

In order to study a solution $x_t(s)$ we should properly describe the domain over which $x_t(s)$ is defined. We first fix some $T>0$. Then for every increasing profile $N_t$ a solution is defined on a domain $D_T$ which can be described in a couple of possible ways. First, we can write $D_T=\{ (s,t) : 0 \leq t \leq T, \quad 0 \leq s \leq N_t \}$. On the other hand, since we want a description of the solution achieved by forward integration in time starting from the initial or boundary data, we may also write $D_T$ as $D_T=D_{T_1}\cup D_{T_2}$, where
\begin{equation*} D_{T_1}:=\{ (s,t) : 0 \leq s \leq N_0 , \, \, 0 \leq t \leq T \} ,\quad D_{T_2}:=\{(s,t) : N_0 <s \leq N_T, \, \,  N^{-1}(s)\leq t \leq T \}.\end{equation*}
\begin{remark} In the definition of inverse $N^{-1}$ of the population function $N_t$ it is important to note that since it is a non decreasing function, that can be constant on intervals, we have used the generalized inverse definition
\begin{equation*} N^{-1}(s):=\inf \{ t \in \mathbb{R}_+ : N_t > s\} . \end{equation*}
\end{remark}
Now that we have defined $D_T$ this way, we may integrate forward in time to get the $x_t(s)$ value
\begin{equation} \label{eq: Proper} x_t(s) \! = \! \! \left\{ \begin{array}{ll}  \! \! \! x_0(s)+\int_0^t \diff t' \frac{1}{N_{t'}}\int_0^{N_{t'}} \diff s' \psi(|x_{t'}(s')-x_{t'}(s)|) (x_{t'}(s')-x_{t'}(s)) , \quad (s,t)\in D_{T_1} , \\ \\ \! \! \!X(N^{-1}(s),s)+\int_{N^{-1}(s)}^t \! \diff t' \frac{1}{N_{t'}}\! \int_0^{N_{t'}} \! \diff s' \psi(|x_{t'}(s')-x_{t'}(s)|) (x_{t'}(s')-x_{t'}(s)), \quad (s,t)\in D_{T_2}. \end{array} \right.\end{equation}

This description allows us to introduce a ``mild'' notion of solution which does not require differentiability (in time) of $x_t(s)$. In fact, we define both the classical and a mild solution (in these sense that was introduced in \cite{BlHeTs2}) in the following manner
\begin{definition} Let $T>0$ and some given population growth profile $N_t$ for $0\leq t \leq T$. Then we have:
\\ (i) A \textit{classical} solution of \eqref{eq: conses_symm} is a solution $x_t(s)$ that satisfies \eqref{eq: conses_symm} for every $(s,t) \in D_T$, with initial profile $x_0(s)$ that also satisfies \eqref{eq: Mov_Bound_Cond}. Thus, a classical solution is time-differentiable everywhere in $D_T$.
\\ (ii) A \textit{mild} solution of \eqref{eq: conses_symm} is a solution that satisfies \eqref{eq: Proper} for all $(s,t) \in D_T$.
\end{definition}

In the following result we show that when the interaction function $\psi(r)$ is Lipschitz continuous, then as one would expect a classical solution exists.
\begin{theorem} \label{thm: Well_posedness1} Assume an increasing population $N_t$ governed by \eqref{eq: growth}. Let $x_0 \in L^{\infty}([0,N_0];\mathbb{R}^d)$, and let $\psi(\cdot)$ be a Lipschitz interaction kernel with a Lipschitz constant $L_\psi>0$. Let also $X(t, N_t)$ be a given boundary condition for the states of incoming agents that satisfy \eqref{eq: X(t)_cond}. Then, for any $T>0$, there exists a unique classical solution $x_t(s)$ of \eqref{eq: conses_symm} for $(s,t) \in D_T$, with initial condition $x_0(s)$, that satisfies the boundary condition \eqref{eq: Mov_Bound_Cond} for $0 \leq t \leq T$. \end{theorem}

The next natural question is the study of the asymptotics of system \eqref{eq: conses_symm}. The main feature of the model, is that unlike the system with a fixed population, here the average defined by $m_1(t):=\frac{1}{N_t}\int_0^{N_t} x_t(s)\, \diff s$ is not conserved due to the influx of new agents with a prescribed state. There are two main options for population growth depending on whether $N_t$ remains uniformly bounded in time ($N_\infty < \infty$), or it tends to infinity, i.e. $N_\infty=\infty$. The asymptotic behavior when $N_\infty < \infty$ resembles one of a system with a fixed number of agents, and the opinions tend converge to clusters or a global consensus.

When $N_\infty=\infty$, it is clear that the influence of the opinions of incoming agents is going to play the major role in the asymptotic behavior of the system. The determining factor of the long term asymptotic dynamics is whether the average $m_1(t)$ will converge to $X(t, N_t)$. We show that the following condition
\begin{equation*} \frac{1}{N_t}\int_0^t X(s, N_s) \dot{N}_s \, \diff s \xrightarrow{ t \to \infty }
 X(t, N_t), \qquad \qquad (C1) \end{equation*}
is necessary and sufficient for $m_1(t) \xrightarrow{ t \to \infty } X(t, N_t)$.
We may understand condition $(C1)$ better if we assume that $X(t, N_t)$ is differentiable in time and perform a simple integration by parts to write $(C1)$ as
\begin{equation*} \frac{1}{N_t} \int_0^t \dot{X}(s, N_s) N_s \, \diff s \xrightarrow{ t \to \infty } 0 . \qquad \qquad (C1') \end{equation*}
The time derivative $\dot{X}(t, N_t)$ is $\dot{X}(t, N_t)=X_t(t, N_t)+X_N(t, N_t)\dot{N}_t$. The new condition $(C1')$  indicates the role that the rate of change of $X(t, N_t)$ plays for $(C1)$ to hold. That being said, we should point that $X(t, N_t)$ approaching a constant value, or similarly $\dot{X}(t, N_t)$ diminishing as $t \to \infty$ is neither a necessary nor a sufficient condition for $(C1)$ or $(C1')$ to hold as we also have the dependence on the population growth rate as $N_t \to \infty$.
Moreover, condition $(C1)$ implies the vanishing of the total variance defined by $\pazocal{V}(x_t(\cdot)):=\frac{1}{N_t} \int_0^{N_t}|x_t(s)-m_1(t)|^2\, \diff s$, i.e. that $\pazocal{V}(x_t(\cdot)) \xrightarrow{ t \to \infty } 0$.

We have the following result that holds for communication weights of both \textbf{Types I} and \textbf{II}:
\begin{theorem} \label{thm: Convergence1} Assume an increasing population $N_t$ that evolves according to \eqref{eq: growth}. Let  $x_t(s)$ be a solution of \eqref{eq: conses_symm}, for $t>0$ and $s \in [0, N_t]$, for the given population $N_t$. Assume also that $X(t, N_t)$ is a continuous profile for the states of incoming population. We have the following regarding the long time behavior of $x_t(s)$:
\\
(i) If $N_\infty < \infty$ the average $m_1(t)$ converges to some explicit value $m_*$, i.e.
\begin{equation*} m_1(t) \xrightarrow{ t \to \infty } m_*:=\frac{N_0}{N_\infty}m_1(0)+\frac{1}{N_\infty}\int_0^\infty X(s, N_s) \dot{N}_s \, \diff s. \end{equation*}
In particular, opinions form $J \in \mathbb{N}$ disjoint clusters $ C_j \subset [0,N_{\infty}] $, for $j=1, \ldots, J$  s.t. $\sum \limits_{j=1}^J  m(C_j)=N_\infty$. The clusters are characterized by the following property : For almost all pairs $(s,s') \in C_i \times C_j$, then
\begin{equation*} |x_t(s)-x_t(s')| \xrightarrow{t \to \infty} 0 \quad \text{for} \quad i=j, \qquad \text{and} \quad  \liminf \limits_{t \to \infty} |x_t(s)-x_t(s')| \geq 1 \quad \text{if} \quad i\neq j . \end{equation*}
\begin{comment}In particular, opinions form $J \in \mathbb{N}$ clusters with values $\{ \bar{x}_1,\ldots , \bar{x}_J\}$ on the sets $\{C_1, \ldots , C_J \}$, i.e. for almost every $s \in [0,N_\infty]$ then $s \in C_j$ (for some $j$) and $x_t(s)\to \bar{x}_j$. The properties of the sets $C_j$ and values $\bar{x}_j$ are
\begin{equation*} \sum \limits_{j=1}^J  m(C_j)=N_\infty , \qquad \frac{1}{N_\infty} \sum \limits_{j=1}^J m(C_j)\bar{x}_j =m_* .\end{equation*}   \end{comment}
\\
(ii) If $N_\infty=\infty$ and $(C1)$ holds we then have that $m_1(t)  \xrightarrow{ t \to \infty } X(t, N_t) $. Furthermore, the variance vanishes, i.e. $\pazocal{V}(x_t(\cdot))  \xrightarrow{ t \to \infty } 0$. Moreover, when $X(t, N_t)$ has a constant value $X_c$, then all opinions eventually converge to that value and the total variance dissipates at a specific rate, i.e. \begin{equation*} \pazocal{V}(x_t(\cdot))\leq (\pazocal{V}(x_0(\cdot))+|m_1(0)-X_c|^2) \, e^{-\int_0^t b(s, N_s) \, \diff s}.\end{equation*}
\end{theorem}

When the interaction topology is such that all agents interact with each other and $\psi(r)>0$, for $r>0$, we have a stronger convergence result. In this case, when $N_\infty < \infty$ we have convergence to global consensus (mono-cluster formation), but more importantly when $N_\infty=\infty$ condition $(C1)$ is not necessary for a vanishing $\pazocal{V}(x_t(\cdot))$ as long as we have a population growth that is subexponential at infinity, i.e.
\begin{theorem} \label{thm: Convergence2} Let $x_t(s)$ be a solution of \eqref{eq: conses_symm}, for $t>0$ and $s \in [0 ,N_t]$, for the given population $N_t$. Assume also a positive influence function $\psi(r)>0$. The following results hold for the asymptotic behavior of $x_t(s)$:
\\
(i) If $N_\infty < \infty$ and $b(t, N_t)\in L^{\infty}([0,\infty);\mathbb{R}_+)$ all the opinions converge in Lebesgue measure to a consensus value which
is the long time average $m_*$, i.e. for every $\epsilon>0$
\begin{equation*} \lim \limits_{t \to \infty} m(s \in [0,N_{\infty}) \, : \, |x_t(s)-m_*|> \epsilon)=0  . \end{equation*}
\\
(ii) If $N_\infty = \infty$ and $b(t,N_t) \to 0$ as $t \to \infty$, then it follows that $\pazocal{V}(x_t(\cdot)) \xrightarrow{ t \to \infty } 0$. Moreover, if $b(t,N_t)=O\left(\frac{1}{t^\alpha}\right)$ for $0< \alpha \leq 1$, then also $\pazocal{V}(x_t(\cdot))=O\left(\frac{1}{t^\alpha}\right)$ as $t \to \infty$. \end{theorem}

To avoid the curse of dimensionality it is extremely useful to introduce a macroscopic description which is based on a probabilistic description of the states of the system. For this we introduce a measure $f_t(x)\in \pazocal{P}(\mathbb{R}^d)$ which represents the probability density of agents having state $x$ at time $t$.
This probability measure is the solution to the following transport equation
\begin{equation} \label{eq: kinet} \partial_t f_t(x) + \nabla \cdot (V[f_t](x) f_t(x))= h[f_t](x) , \end{equation}
which is paired with  an initial probability profile $f_0(x)$ for $x \in \mathbb{R}^d$. The interaction Kernel $V[f_t](x)$ and the source term $h[f_t](x)$ are defined by
\begin{equation} \label{eq: inter_kernel_source} V[f_t](x):=\int_{\mathbb{R}^d} \psi(|y-x|) (y-x) f_t(y) \, \diff y , \quad h[f_t](x):= b(t,N_t) \left( \delta_{X(t, N_t)}(x)-f_t(x) \right). \end{equation}
The convolution kernel term $V[f_t](x)$ in \eqref{eq: inter_kernel_source} is the standard transport term in the classical symmetric consensus model. On the other hand, the source term $h[f_t](x)$ has a rather simple representation in that it is proportional to the growth rate $b(t,N_t)$ (population growth/total population) and it also has an ``influx'' $\delta_{X(t, N_t)}(x)$ and an ``outflux'' $f_t(x)$ part. The influx $\delta_{X(t, N_t)}(x)$ is a Dirac delta mass centered around $X(t, N_t)$ and the outflux is a term that accommodates for the extra mass added to the system via the boundary.

We shall give a result on the uniqueness of solutions of \eqref{eq: kinet} by proving a stability estimate in the 1-Wasserstein metric topology. The importance
of this metric is that it metrizes the weak-$^*$ convergence on compact sets. For $p \in [1, +\infty)$ the
p-Wasserstein distance between two probability measures $\mu , \nu \in \pazocal{P}_p(\mathbb{R}^d)$ (space of probability measures with a finite $p$ moment) is defined as
\begin{equation*} W_p(\mu ,\nu )=\inf \limits_{\gamma \in \Gamma(\mu, \nu)}\left( \int |x-x'|^p \, \diff \gamma(x, x') \right)^{\frac{1}{p}} ,\end{equation*}
where $\Gamma(\mu, \nu)$ is the set of all couplings between $\mu$ and $\nu$, i.e. the set of all joint probability measures with marginals $\mu$ and $\nu$ s.t.
\begin{equation*} \mu(A)=\gamma(A \times \mathbb{R}^d), \qquad \nu(B)=\gamma(\mathbb{R}^d \times B), \quad A, B \subset \mathbb{R}^d .\end{equation*}
The Kantorovich-Rubinstein duality theorem (see e.g. \cite{Vil}) states that in the special case $p=1$ the 1-Wasserstein metric of two compactly supported measures $\mu , \nu$ is
\begin{equation} \label{eq: W_1} W_1(\mu ,\nu)= \sup \limits_{\|\nabla \phi\|_{L^{\infty}(\mathbb{R}^d)} \leq 1} \int \phi(x) \, \diff (\mu -\nu)(x) , \end{equation}
where the supremum is taken over all $\phi : \mathbb{R}^d \to \mathbb{R}$ with Lipschitz constant 1. The norm $\|\nabla \phi\|_{L^{\infty}(\mathbb{R}^d)}$ in \eqref{eq: W_1} is defined as $\|\nabla \phi\|_{L^{\infty}(\mathbb{R}^d)}:= \max \limits_{1 \leq j \leq d} \esssup |\partial_{x_j}\phi |$. It is well known that $\pazocal{P}_1(\mathbb{R}^d)$ endowed with the $W_1(\cdot, \cdot)$ distance is a complete metric space.

The following is the main result on the well-posedness of \eqref{eq: kinet}
\begin{theorem} \label{thm: Well_posedness2} Assume a Lipschitz interaction function $\psi(r)$ with Lipschitz constant $L_{\psi}>0$. Let also $T>0$ and $f_0(x) \in \pazocal{P}_1(\mathbb{R}^d)$ an initial profile with bounded support. Then there exists a unique measure-valued solution that satisfies \eqref{eq: kinet} in the weak sense. Furthermore, if $f_t$ and $\widetilde{f}_t$ are two measure-valued solutions with initial conditions $f_0$ and $\widetilde{f}_0$ respectively, then the following stability estimate holds
\begin{equation*} \sup \limits_{t \in [0,T]}W_1(f_t,\widetilde{f}_t) \leq C_T W_1(f_0,\widetilde{f}_0) , \end{equation*} where $C_T$ is a positive constant that depends on $L_\psi$, $T$, $X_B$, $\|b\|_{L^{\infty}([0,T))}$, and the supports of initial profiles $f_0$ and $\widetilde{f}_0$. \end{theorem}

In the same lines as with Theorems \ref{thm: Convergence1} and \ref{thm: Convergence2} we prove the following result.

\begin{theorem} \label{thm: Convergence_kin} Assume a measure solution $f_t(x)$ of the kinetic problem \eqref{eq: kinet} with compactly supported initial data $f_0(x) \in \pazocal{P}_1(\mathbb{R}^d)$. The following convergence results hold:  \\
(i) Let $N_{\infty} < \infty$, and assume a communication weight $\psi(\cdot)$ which is compactly supported on $[0,1]$. Then the solution concentrates
around a finite set of $J$ points $x_i$ , for $i=1,\ldots, J$ s.t. $|x_i-x_j|\geq 1$ for $1 \leq i \neq j \leq J$, i.e we have that
\begin{equation*} f_t \rightharpoonup^* \sum \limits_{i=1}^J a_i \delta_{x_i} \quad \text{as} \quad t \to \infty , \quad \text{where} \end{equation*}
 \begin{equation*} \sum \limits_{i=1}^J a_i=1  \quad \text{and} \quad \sum \limits_{i=1}^J a_i x_i = M_*:=\frac{N_0}{N_{\infty}}\int_{\mathbb{R}^d} x f_0(x) \, \diff x+\frac{1}{N_{\infty}}\int_0^{\infty} X(s, N_s) \dot{N}_s \, \diff s .\end{equation*}
In the case $\psi(\cdot)>0$, the measure solution $f_t(x)$ concentrates around the value $M_*$ in the sense
\begin{equation*} f_t \rightharpoonup^* \delta_{M_*} \quad \text{as}\quad t \to \infty .\end{equation*}
\\
(ii) For $N_{\infty}= \infty$ assume a general communication weight $\psi(\cdot)$ (compactly supported or everywhere positive). Then if $(C1)$ holds we have $f_t \rightharpoonup^* \delta_{X(t, N_t)} \quad \text{as} \quad t \to \infty$.

If $\psi(\cdot)>0$, then condition $(C1)$ is not necessary for convergence. If $b(t, N_t) \to 0$ then any solution $f_t(x)$ concentrates around $M(t)=\frac{1}{N_t}\int_0^t X(s, N_s) \dot{N}_s \, \diff s$, i.e. $f_t \rightharpoonup^* \delta_{M(t)} \quad \text{as}\quad t \to \infty $.
\end{theorem}

\section{Well-Posedness of the microscopic model}
\label{sec: Well-Posedness}
The structure of the consensus model allows a uniform bound in time of the magnitude of a solution $x_t(s)$:

\begin{lemma} \label{lemma: Bound} Assume a classical solution $x_t(s)\in \mathbb{R}^d$ of \eqref{eq: conses_symm}, with initial conditions $x_0(s)$ for $s \in [0, N_0]$, and incoming opinion profile $X(t, N_t)$ bounded by some $X_B>0$. Let's also define \begin{equation*} \pazocal{R}:=\max \left\{\sup\limits_{0 \leq s \leq N_0} |x_0(s)|, \, X_B \right\} . \end{equation*} Then it follows that $\sup \limits_{0 \leq s \leq N_t}|x_t(s)|\leq \pazocal{R}$ for all $t>0$. \end{lemma}

\begin{proof}[Proof of Theorem \ref{thm: Well_posedness1}]

The proof relies on Banach contraction theorem. We begin by writing the equation in the form $x_t(s)=Kx_t(s)$, where the operator $K$ is described by the branch function \eqref{eq: Proper} for some $T>0$. We define the norm $\|x\|=\sup \limits_{(s,t)\in D_T}|x_t(s)|$ for functions defined on $D_T$.
Our aim is to show that $\|Kx-Ky\|<\| x-y\|$ for any two functions defined on $D_T$.

We have that
\begin{equation*} (s,t)\in D_{T_1}: \quad |Kx_t-Ky_t|(s)= \Bigg|\int_0^t \diff t' \frac{1}{N_{t'}} \int_0^{N_{t'}} \diff s' \, I_{t'}(s,s')\Bigg| \leq \int_0^t \diff t' \frac{1}{N_{t'}} \int_0^{N_{t'}} \diff s' \, |I_{t'}(s,s')| ,\end{equation*}

\begin{equation*} (s,t)\in D_{T_2}: \quad |Kx_t-Ky_t|(s)= \Bigg|\int_{N^{-1}(s)}^t \diff t' \frac{1}{N_{t'}} \int_0^{N_{t'}} \diff s' \, I_{t'}(s,s') \Bigg| \leq \int_{N^{-1}(s)}^t \diff t' \frac{1}{N_{t'}} \int_0^{N_{t'}} \diff s' \, |I_{t'}(s,s')|,\end{equation*}
where
\begin{equation*} I_{t}(s,s')=\psi(|x_t(s')-x_t(s)|)(x_t(s')-x_t(s))-\psi(|y_t(s')-y_t(s)|)(y_t(s')-y_t(s)) .\end{equation*}
This can be written as
\begin{align*} I_{t}(s,s')&=\psi(|x_t(s')-x_t(s)|)\left( (x_t(s')-y_t(s')) - (x_t(s)-y_t(s))\right) \\ &+ \left( \psi(|x_t(s')-x_t(s)|)-\psi(|y_t(s')-y_t(s)|\right) (y_t(s')-y_t(s)) . \end{align*}
Taking the Euclidean norm of $I(t)$ we get
\begin{align*} |I_{t}(s,s')| & \leq |(x_t(s')-y_t(s')) - (x_t(s)-y_t(s))| + |\psi(|x_t(s')-x_t(s)|)-\psi(|y_t(s')-y_t(s)|) \, |y_t(s')-y_t(s)| \\ & \leq |(x_t(s')-y_t(s')) - (x_t(s)-y_t(s))| + L_\psi \, |y_t(s')-y_t(s)| \, ||x_t(s')-x_t(s)|-|y_t(s')-y_t(s)|| \\ & \leq (1+ L_\psi |y_t(s')-y_t(s)|) \, |(x_t(s')-y_t(s')) - (x_t(s)-y_t(s))| \\ & \leq (1+ L_\psi |y_t(s')-y_t(s)|) ( |x_t(s')-y_t(s')|  + |x_t(s)-y_t(s)| ) .
\end{align*}
If for a moment we assume that any solution is bounded uniformly by $\pazocal{R}$ then we get
\begin{equation*} |I_{t}(s,s')|  \leq (1+ 2 L_\psi \pazocal{R}) ( |x_t(s')-y_t(s')|  + |x_t(s)-y_t(s)| ) .\end{equation*}
Working with this we get that for every pair $(t,s)\in D_{T_1}$ and $(t,s)\in D_{T_2}$ then
\begin{equation*} |Kx_t -Ky_t|(s) \leq 2T (1+2L_\psi \pazocal{R}) \|x-y \|\end{equation*}
and taking the supremum over $D_{T_1}\cup D_{T_2}$ we get
\begin{equation*} \|Kx -Ky \|\leq 2T (1+2L_\psi \pazocal{R}) \|x-y \|.\end{equation*}

Thus, choosing $\widetilde{T} \leq (2(1+2L_{\psi}\pazocal{R}))^{-1}$ ensures that $K$ is a contraction on $[0, \widetilde{T}]$. By the Banach contraction theorem there exists a unique solution $x_t \in C([0,\widetilde{T}]; L^{\infty}([0, N_t];\mathbb{R}^d))$. Given this local solution, we now take the values $x_{\widetilde{T}}(s)$ (for $0 \leq s \leq N_{\widetilde{T}}$) as initial conditions and we extend the local solution to $[0,2 \widetilde{T}]$. We repeat until we cover $[0,T]$. We also note that the integrand term of $K$ is continuous map from  $L^{\infty}([0, N_t])$ to $L^{\infty}([0, N_t])$ so $x_t \in C^1([0,\widetilde{T}]; L^{\infty}([0, N_t];\mathbb{R}^d))$.

\end{proof}

\section{Moments of the continuous model}
\label{sec: Moments}
Before we proceed with the definition of moments, we give the following lemma which is helpful for the computation of their evolution. The proof is an elementary exercise in chain differentiation and given in the Appendix
\begin{lemma} \label{lemma: Differ} We assume a population $N_t$ that grows according to \eqref{eq: growth} for some continuous growth rate $b(t, N_t)$. Let's also assume some function $G(t,x)$ that is $C^{1,1}(\mathbb{R}_+,\mathbb{R}^d)$. Finally, we assume that $x_t(s)$ is a solution to \eqref{eq: conses_symm}--\eqref{eq: X(t)_cond} with a continuous in $t$ profile for incoming opinions $X(t, N_t)$. Then, the functional $\pazocal{B}(t)$ defined by $\pazocal{B}(t):=\frac{1}{N_t}\int_0^{N_t} G(t,x_t(s)) \, \diff s$ is differentiable with \begin{equation*} \dot{\pazocal{B}}(t) = -b(t, N_t) \pazocal{B}(t) + b(t, N_t) G(t,X(t, N_t))+\frac{1}{N_t}\int_0^{N_t}(G_t +  G_x \cdot \dot{x}_t(s) )\, \diff s .\end{equation*}
\end{lemma}

We define the first three (weighted) moments by
\begin{equation} \label{eq: Moments} m_0(t):=\frac{1}{N_t} \int_0^{N_t} 1 \, \diff s  , \quad m_1(t):=\frac{1}{N_t} \int_0^{N_t} x_t(s) \, \diff s , \quad  m_2(t):= \frac{1}{N_t} \int_0^{N_t} |x_t(s)|^2 \, \diff s . \end{equation}

\begin{proposition} \label{proposition: moment_Deriv} Let $x_t(s)$ a solution to \eqref{eq: conses_symm} with a boundary condition $X(t, N_t)$. Then the following expressions hold for the moments of $x_t(s)$ defined in \eqref{eq: Moments} (for $t \geq 0$)
\begin{equation} \label{eq: m_iDeriv}
\begin{aligned} (i) \quad \dot{m}_0(t)&=0
\\ (ii) \quad \dot{m}_1(t) &=-b(t, N_t) m_1(t)+b(t, N_t) X(t, N_t)
\\ (iii) \quad \dot{m}_2(t) &=-b(t, N_t) m_2(t)+b(t, N_t) |X(t, N_t)|^2 - D(x_t(\cdot)),
\end{aligned} \end{equation}
where the expression for the dissipation functional $D(x_t(\cdot))$ is
\begin{equation*}  D(x_t(\cdot)):=\frac{1}{N_t^2}  \int_0^{N_t} \int_0^{N_t} \psi(|x_t(s)-x_t(s')|)|x_t(s)-x_t(s')|^2 \,\diff s' \diff s . \end{equation*}
\end{proposition}

\begin{proof} The first equation in \eqref{eq: m_iDeriv} is trivial since $m_0(t)=1$ for all $t \geq 0$.
Also, the symmetry in the communication interaction $\psi(|x(s')-x(s)|)$ w.r.t $s$ and $s'$ implies the identities
\begin{align} \label{eq: Symm} \frac{1}{N_t} \int_0^{N_t} \dot{x}_t(s) \, \diff s=0 , \qquad \frac{1}{N_t} \int_0^{N_t} \frac{d}{dt}|x_t(s)|^2 \, \diff s=-D(x_t(\cdot)) .\end{align} Consequently $\dot{m}_1(t)$ and $\dot{m}_2(t)$
follow from a direct application of Lemma \ref{lemma: Differ} and \eqref{eq: Symm}.
\end{proof}

We note that in contrast to the classical symmetric consensus model (where we have conservation of the average opinion) here the average $m_1(t)$ depends both on the growth rate $b(t,N_t)$ as well as the profile for incoming opinions $X(t, N_t)$. In the corollary that follows we solve w.r.t the first moment $m_1(t)$ and show its asymptotic behavior. Consequently we show why in the case  $N_\infty=\infty$ condition $(C1)$ is equivalent to $m_1(t) \xrightarrow{t \to \infty} X(t, N_t)$.

\begin{corollary} \label{corollary: m_1}
Let $x_t(s)$ for $s \in [0,N_t)$ be a solution to the symmetric consensus model \eqref{eq: conses_symm} with boundary data $X(t, N_t)$. When $N_\infty=\infty$ we have that
\begin{equation} \label{eq: C1} m_1(t) \xrightarrow{ t \to \infty } \frac{1}{N_t}\int_0^t X(s, N_s) \dot{N}_s\, \diff s  .\end{equation}
Hence, $(C1)$ is necessary and sufficient for $m_1(t) \xrightarrow{ t \to \infty } X(t, N_t)$. Furthermore, if we assume that $X(t, N_t)$ is differentiable in time we have
\begin{equation} \label{eq: C2} \left|m_1(t)-X(t, N_t)+ \frac{1}{N_t}\int_0^t \dot{X}(s, N_s)N_s \, \diff s \right| = |m_1(0)-X(0, N_0)| \, \frac{N_0}{N_t} .\end{equation}
When $N_\infty<\infty$ then
\begin{equation} \label{eq: m_1Finite} m_1(t) \xrightarrow{ t \to \infty } \frac{N_0}{N_\infty}m_1(0)+\frac{1}{N_{\infty}}\int_0^{\infty} X(s, N_s) \dot{N}_s \, \diff s.\end{equation}
\end{corollary}

\begin{proof}
We can solve the $m_1(t)$ first order equation in \eqref{eq: m_iDeriv} with the help of the integrating factor $e^{\int_0^t b(s, N_s) \, \diff s}$ and keeping in mind that both $b(t, N_t)$ and $X(t, N_t)$ are given, i.e.
\begin{equation} \label{eq: m_1Sol1}  m_1(t)=e^{-\int_0^t b(s, N_s) \, \diff s} m_1(0) +  \int_0^t b(s, N_s)X(s, N_s)e^{-\int_s^t b(s', N_{s'}) \, \diff s'} \, \diff s \end{equation}
The assumption on the differentiability of $X(t, N_t)$ allows integration by parts, i.e. \begin{equation} \label{eq: m_1Sol2} m_1(t)=e^{-\int_0^t b(s, N_s) \, \diff s} (m_1(0)-X(0, N_0)+X(t, N_t)-\int_0^t \dot{X}(s, N_s)e^{-\int_s^t b(s',N_{s'}) \, \diff s'} \, \diff s .  \end{equation}
Observing that $e^{-\int_0^t b(s, N_s) \, \diff s}=\frac{N_0}{N_t} , \quad  e^{-\int_s^t b(s', N_{s'}) \, \diff s'}=\frac{N_s}{N_t}$ and we get that \eqref{eq: C1}--\eqref{eq: C2}  follow  from \eqref{eq: m_1Sol1}--\eqref{eq: m_1Sol2}.

When $N_\infty < \infty$ we observe that we can write \begin{equation*}
\int_0^t b(s, N_s)X(s, N_s)e^{-\int_s^t b(s', N_{s'}) \, \diff s'} \, \diff s =\frac{1}{N_t}\int_0^t
X(s, N_s) \dot{N}_s \, \diff s ,\end{equation*}
and show by taking $t \to \infty$ that
\begin{equation*} \int_0^t b(s, N_s)X(s, N_s)e^{-\int_s^t b(s', N_{s'}) \, \diff s'} \, \diff s \xrightarrow{ t \to \infty } \frac{1}{N_{\infty}} \int_0^{\infty} X(s, N_s) \dot{N}_s \, \diff s .\end{equation*}
The last integral can be shown to exist due to $N_\infty < \infty$, and thus \eqref{eq: m_1Finite} follows.
\end{proof}

\textbf{Examples}: We give some examples of pairs of $X(t, N_t)$ and $N_t$ and check the validity of $(C1)$--$(C1')$.
\begin{enumerate}
\item \textbf{Constant $X(t, N_t)$}. Every constant $X(t, N_t)=X_c$ trivially satisfies $(C1)$ for every rate growth $b(t, N_t)$. In this case the solution of \eqref{eq: m_1Sol2} takes the form
\begin{equation*} m_1(t)=X_c +(m_1(0)-X_c) e^{-\int_0^t b(s, N_s) \, \diff s} .\end{equation*}
Trivially, when $N_\infty =\infty$ then $m_1(t) \xrightarrow{ t \to \infty } X_c$, whereas when $N_\infty < \infty$ then $m_1(t) \xrightarrow{ t \to \infty } \frac{N_0}{N_\infty} m_1(0)+\left( 1 - \frac{N_0}{N_\infty} \right)X_c$.

If $X(t, N_t)$ is eventually constant, i.e. $X(t, N_t)=X_c$ for $t \geq T_0$ (for some $T_0>0$)  then  $(C1)$ holds when $N_\infty = \infty$. When $N_\infty < \infty$, then
\begin{equation*} \frac{1}{N_t}\int_0^t \dot{X}(s, N_s) N_s \, \diff s  \xrightarrow{ t \to \infty } \frac{1}{N_\infty} \int_0^{T_0} \dot{X}(s, N_s) N_s \, \diff s .\end{equation*}
\item \textbf{Exponential population growth}. When the population grows exponentially, the average $m_1(t)$ may oscillate even for a Lipschitz (in time) $X(t, N_t)$. Consider e.g. $b(t, N_t)=1$ and $X(t, N_t)=\sin t$, then we have
\begin{equation*} \frac{1}{N_t}\int_0^t X(s, N_s)\dot{N}_s \, \diff s= \int_0^t \sin s e^{-(t-s)}\, \diff s =\frac{1}{2} (\sin t -\cos t)+\frac{1}{2}e^{-t} .\end{equation*}
It is evident that this integral has values that oscillate between some range as $t \to \infty$ and it clearly has no fixed limit. Thus, $m_1(t)\sim \frac{1}{2}(\sin t +\cos t)$ as $t \to \infty$.
\item \textbf{Extremely fast population growth}. Assume a Lipschitz $X(t, N_t)$, and hyper-exponential population growth $b(t, N_t)\geq t^\alpha$, for $\alpha>0$. We have
\begin{equation*} \Bigg| \frac{1}{N_t}\int_0^t \dot{X}(s, N_s) N_s \, \diff s \Bigg| \leq \frac{C}{1+\alpha} e^{-t^{1+\alpha}} \int_0^t e^{s^{1+\alpha}} \, \diff s .\end{equation*}
We remark here the following integral
\begin{equation*} F(p,x)=e^{-x^p}\int_0^x e^{t^p} \, dt , \qquad p \geq 1 ,\end{equation*}
which is the generalized Dawson function (Dawson integral is for $p=2$) and has the asymptotic expansion $F(p,x)\sim ax \sum \limits_{n\geq 0}c_n / x^{p(n+1)}$. This implies that when $p \geq 1$, then $F(p,x) \to 0$ as $x \to \infty$.
\item \textbf{Incoming opinion profile $X(N_t)$ is only a function of the population}.
In this case, if we assume the bound $|X_N(N_t)|\leq C N_t^{-1-\epsilon}$ for some $\epsilon>0$ ($\forall t >0$), then $(C1')$ holds. Note that in this case we don't have to assume anything on the $N_t$ growth. Indeed, we have that
\begin{equation*} \left| \frac{1}{N_t}\int_0^t \dot{X}(N_s) N_s \, \diff s \right|=\frac{1}{N_t} \int_0^t |X_N(N_s)| \dot{N}_s N_s \, \diff s \leq \frac{1}{N_t} \int_0^t \dot{N}_s N_s^{-\epsilon} \, \diff s =\frac{C}{1-\epsilon}N_t^{-\epsilon} \xrightarrow{ t \to \infty } 0. \end{equation*}
\end{enumerate}

The following easy lemma will prove essential later.
\begin{lemma} \label{lemma: C1'_condition} Assume a growing population $N_t$ with $N_\infty =\infty$. Let also a function $g(t)\in L^{\infty}([0,\infty))$ s.t. $g(t) \xrightarrow{ t \to \infty } 0$. Then it follows that \begin{equation*} \frac{1}{N_t}\int_0^t g(s) \dot{N}_s \, \diff s \xrightarrow{ t \to \infty } 0. \end{equation*} \end{lemma}

Direct consequences of this lemma are the following two simple facts :
\\ (i) A sufficient condition for the $(C1')$ condition is $\frac{\dot{X}(t, N_t)}{b(t, N_t)}\xrightarrow {t \to \infty}0$. Particularly, many of the examples mentioned above can be tested using this criterion.
\\ (ii) Another consequence is that if we consider an exponential growth with $N_t=e^{\lambda t}N_0$ and an $L^\infty$ function s.t. $g(t)\xrightarrow {t \to \infty}0$, then $\int_0^t g(s)e^{-\lambda(t-s)}\, \diff s \xrightarrow{ t \to \infty } 0$. As a matter of fact, if $g(t)=O\left(\frac{1}{t^\alpha}\right)$ for $0< \alpha $, it follows by a simple application of L'Hospitals rule that $\int_0^t g(s)e^{-\lambda(t-s)}\, \diff s=O\left(\frac{1}{t^\alpha}\right)$ as $t \to \infty$.

\section{Convergence and Clustering}
\label{sec: Convergence}
We now move to the study of the long time asymptotics of our model. The \textbf{total variance} of the opinion profile $x_t(\cdot)$ is the functional \begin{equation}\label{eq: Liapunov_X1} \pazocal{V}(x_t(\cdot)):=\frac{1}{2N_t^2}\int_0^{N_t} \int_0^{N_t} |x_t(s')-x_t(s)|^2 \, \diff s' \, \diff s . \end{equation}
For later use we also define the square distance between the average $m_1(t)$ and $X(t, N_t)$,
\begin{equation*} \pazocal{M}_1(t):=|m_1(t)-X(t, N_t) |^2 . \end{equation*} $\pazocal{M}_1(t)$ depends explicitly on $m_1(0)$ and the profiles of $X(t, N_t)$ and $N_t$, for $0 \leq t <\infty$. With the help of \eqref{eq: m_iDeriv} and assuming that $X(t, N_t)$ is differentiable we get
\begin{equation} \label{eq: Distance_X-m_1}  \dot{\pazocal{M}}_1(t) =- 2 b(t, N_t)\pazocal{M}_1(t) -2 \dot{X}(t, N_t) \cdot (m_1(t)-X(t, N_t)) . \end{equation}
We give the following result regarding the behavior of $\pazocal{V}(x_t(\cdot))$.
\begin{proposition} \label{proposition: V_deriv} Assume a smooth solution $x_t(\cdot)$ of problem \eqref{eq: conses_symm}
with a boundary profile $X(t, N_t)$. The time derivative of \eqref{eq: Liapunov_X1} along this solution is
\begin{equation} \label{eq: Liapunov_m1_Derivative} \dot{\pazocal{V}}(x_t(\cdot))=-b(t, N_t) \, \pazocal{V}(x_t(\cdot)) +b(t, N_t) \pazocal{M}_1(t) - D(x_t(\cdot)) . \end{equation}
\end{proposition}

\begin{proof}
We start by noting that the variance $\pazocal{V}(x_t(\cdot))$ can be written as a single integral with the help with the help of the average $m_1(t)$ as
\begin{equation*} \pazocal{V}(x_t(\cdot))= \frac{1}{N_t} \int_0^{N_t}
|x_t(s)-m_1(t)|^2 \, \diff s .\end{equation*}
Differentiating this with the help of Lemma \ref{lemma: Differ} yields
\begin{align*} \dot{\pazocal{V}}(x_t(\cdot)) = -b(t, N_t)\pazocal{V}(x_t(\cdot)) + b(t, N_t)\pazocal{M}_1(t) +\frac{2}{N_t} \int_0^{N_t} (\dot{x}_t(s)-\dot{m}_1(t)) \cdot (x_t(s)-m_1(t)) \, \diff s  .\end{align*}
The last integral is the dissipation term $-D(x_t(\cdot))$ as we can see from \begin{align*} \frac{2}{N_t} \int_0^{N_t} (\dot{x}_t(s)-\dot{m}_1(t)) \cdot (x_t(s)-m_1(t)) \, \diff s &\stackrel{\eqref{eq: Moments}}{=}\frac{2}{N_t} \left( \int_0^{N_t} \dot{x}_t(s)\cdot x_t(s) \, \diff s - m_1(t) \cdot \int_0^{N_t} \dot{x}_t(s) \, \diff s \right) \\ & \stackrel{\eqref{eq: Symm}}{=}\frac{1}{N_t}\int_0^{N_t} \frac{d}{dt}|x_t(s)|^2\, \diff s \stackrel{\eqref{eq: Symm}}{=}-D(x_t(\cdot)),
\end{align*}
which concludes the proof of \eqref{eq: Liapunov_m1_Derivative}.
\end{proof}

\begin{remark}
Another choice is the functional that measures the variance around $X(t, N_t)$, i.e.
\begin{equation} \label{eq: Liapunov_X} \pazocal{V}_{X}(x_t(\cdot)):= \frac{1}{N_t} \int_0^{N_t}
|x_t(s)-X(t, N_t)|^2 \, \diff s . \end{equation}
This functional is related to $\pazocal{V}(x_t(\cdot))$ in that $\pazocal{V}_{X}(x_t(\cdot))=\pazocal{V}(x_t(\cdot))+\pazocal{M}_1(t)$. Thus, using \eqref{eq: Distance_X-m_1} we have that
\begin{equation} \label{eq: Liapunov_X_Derivative}\dot{\pazocal{V}}_{X}(x_t(\cdot))=-b(t,N_t) \, \pazocal{V}_{X}(x_t(\cdot)) -2 \dot{X}(t, N_t) \cdot (m_1(t)-X(t, N_t)) -D(x_t(\cdot)) .\end{equation}
The two functionals $\pazocal{V}_{X}(x_t(\cdot))$ and $\pazocal{V}(x_t(\cdot))$ are by definition very closely related and can both offer insight on the long time asymptotics of the continuous model. Note in particular that when $X(t, N_t)=X_c$ then $\pazocal{V}_{X}(x_t(\cdot))$ is dissipative unlike $\pazocal{V}(x_t(\cdot))$.
\end{remark}
With the help of the functionals we introduced we are  in position to give the main convergence results. We treat the cases $N_\infty < \infty$ and $N_\infty = \infty$ separately.

\subsection{The case $N_\infty < \infty$}
In this case, regardless of how fast we approach the value of $N_\infty$, the estimate we get is practically similar to the estimate for a fixed population $N_\infty$ for very large times. To be more precise we have the following control of the dissipation $D(x_t(\cdot))$
\begin{proposition} \label{proposition: Energy_bound} If we assume a population growth profile with
$N_\infty < \infty$, we have that \begin{equation} \label{eq: Cluster} \int_0^\infty D(x_{t}(\cdot))\, \diff t < \infty . \end{equation}
\end{proposition}
\begin{proof} It follows from \eqref{eq: Liapunov_m1_Derivative} that
\begin{equation*} D(x_{t}(\cdot)) \leq -\dot{\pazocal{V}}(x_t(\cdot)) + b(t,N_t) \pazocal{M}_1(t).\end{equation*}
We know that $\pazocal{M}_1(t) \leq C$ for some constant $C>0$, and also $\int_0^{\infty}b(t, N_t) \, \diff t < \infty$ so \eqref{eq: Cluster} follows.
\end{proof}
The condition \eqref{eq: Cluster} depending on the choice of the influence function $\psi(r)$ implies either the convergence of all opinions to a consensus value $x_\infty$ (for $\psi(r)>0$ for $r>0$) or the formation of different opinion clusters when $\psi(r)$ is compactly supported.

\subsubsection{Clustering}

We start with the case of a compactly supported influence $\psi(r)$ on $[0,1]$ which is the most interesting setting. We maintain that in the case of discrete agents the symmetric case is trivial, and the case of the Krause non-symmetric interaction function has been studied extensively in \cite{JabMot}. Here, we can take advantage of the symmetric interaction but the continuous space of agents requires more delicate analysis.

\begin{proposition} \label{proposition: clustering} Let $x_t(s)$ be the solution to the consensus model \eqref{eq: conses_symm} with any measurable boundary condition $X(t, N_t)$ and $N_\infty < \infty$. Assume also that the influence function $\psi(r)$ is compactly supported on $[0,1]$. Then for almost all pairs of agents $(s,s')\in [0,N_\infty) \times [0, N_\infty)$ we have \begin{equation} \label{eq: Cluster2} |x_t(s)-x_t(s')| \xrightarrow{t \to \infty} 0  \quad \text{or} \quad \liminf \limits_{t \to \infty} |x_t(s)-x_t(s')| \geq 1 . \end{equation}
\end{proposition}

\begin{proof}
Let's introduce the set \begin{equation*} B = \{(s,s')\in [0,N_\infty)^2  : |x_t(s)-x_t(s')| \xrightarrow{t \to \infty} 0  \quad \text{or} \quad \liminf \limits_{t \to \infty} |x_t(s)-x_t(s')| \geq 1 \} . \end{equation*} Our aim is to show that the complement $B^c$ has Lebesgue measure zero on the real plane.

We fix some $\epsilon>0$ and introduce the set
\begin{equation*} A_\epsilon =\{ (s,s') : \exists \, \, \text{a real sequence} \, \, t_n \to \infty, \quad  \text{s.t.} \quad |x_{t_n}(s)-x_{t_n}(s')| \in [\epsilon ,1-\epsilon) \} .\end{equation*}
This is the set that contains all the pairs $(s,s')$ of agents for which there exists some sequence $t_n \to \infty$ (which may depend on the choice of $s$, $s'$) for which $x_{t_n}(s)$ and $x_{t_n}(s')$ are distanced by $\geq \epsilon$. We can show that $m(A_\epsilon)=0$ for every $\epsilon > 0$ using estimate \eqref{eq: Cluster}.

Finally, we can see from the definitions of $B$ and $A_\epsilon$ that they are related in that $B^c =\cup_{\epsilon>0} A_\epsilon$ . Since $m(A_\epsilon)=0$ for all $\epsilon>0$ we have that $m(B^c)=0$ and \eqref{eq: Cluster2} follows.
\end{proof}

This result implies the existence of a finite number of clusters $C_j$ (for $1 \leq j \leq J$). In fact, we can define an equivalence relation $\sim$ (see \cite{JabMot})
\begin{equation*} s \sim s' \qquad \text{iff} \qquad |x_t(s)-x_t(s')| \xrightarrow{t \to \infty} 0 .\end{equation*}

\begin{corollary} \label{corollary: clustering} Almost every $s \in [0,N_\infty)$ belongs to some cluster $C_j$ ( $1 \leq j \leq J$). In other words there is only a subset of agents with measure $0$ that are not in any cluster. \end{corollary}

\begin{comment}
\subsubsection{Cluster convergence}
We define the cluster center of the i'th cluster by
\begin{equation*} y_i(t)=\frac{1}{m(C_i)}\int_{C_i}x_t(s) \, \diff s .\end{equation*}
Differentiating we get
\begin{equation*} \dot{y}_i(t)=\frac{1}{m(C_i)}\int_{C_i} \dot{x}_t(s)\, \diff s=\frac{1}{m(C_i)}\int_{C_i} \diff s \frac{1}{N_t}\int_0^{N_t} \diff s' \psi(|x_t(s')-x_t(s)|) (x_t(s')-x_t(s)) \end{equation*}
\end{comment}

\subsection{Case $N_\infty = \infty$}
In this case, the major criterion is whether $m_1(t)$ converges to $X(t, N_t)$, i.e.

\begin{proposition} \label{proposition: C1_criterion} Assume a population growth profile with $N_\infty = \infty$. Assume also that $x_t(s)$ is a solution to \eqref{eq: conses_symm} with a  boundary condition $X(t, N_t)$ that satisfies $(C1)$. Then, it follows that $\pazocal{V}(x_t(\cdot)) \xrightarrow{ t \to \infty } 0$. In particular, when $X(t, N_t)=X_c$ is constant  we have the following estimate for $\pazocal{V}_{X} (x_t(\cdot))$ \begin{equation} \label{eq: V_Xconst} \pazocal{V}_{X} (x_t(\cdot)) \leq \pazocal{V}_{X}(x_0(\cdot))e^{-\int_0^t b(s, N_s)\, \diff s}.\end{equation} \end{proposition}

\begin{proof}
We start with the general case of a
\textbf{boundary condition $X(t, N_t)$ satisfying $(C1)$.}
Using \eqref{eq: Liapunov_m1_Derivative} and ignoring the $D(x_t(\cdot))$ term we have
\begin{equation*} \pazocal{V}(x_t(\cdot)) \leq \pazocal{V}(x_0(\cdot)) e^{-\int_0^t b(s, N_s)\, \diff s} +\int_0^t  \pazocal{M}_1(s) \frac{d}{\diff s}\left(e^{-\int_s^t b(s', N_{s'})\, \diff s'}\right) \, \diff s.\end{equation*}
We have already seen by Corollary 1 that $(C1)$ implies $\pazocal{M}_1(t) \xrightarrow{ t \to \infty } 0$. Since $N_\infty=\infty$, following the proof of Lemma \ref{lemma: C1'_condition} (see Appendix) we get that $ \frac{1}{N_t}\int_0^t  \pazocal{M}_1(s) \dot{N}_s \, \diff s \xrightarrow{ t \to \infty } 0$ which yields $\pazocal{V}(x_t(\cdot)) \xrightarrow{ t \to \infty } 0$.

For a \textbf{constant condition $X(t, N_t)=X_c$} we can work with either $\pazocal{V}(x_t(\cdot))$ or $\pazocal{V}_X(x_t(\cdot))$. We solve \eqref{eq: Distance_X-m_1} to get \begin{equation*} \pazocal{M}_1(t) =\pazocal{M}_1(0) e^{-2\int_0^t b(s, N_s) \, \diff s} .\end{equation*}
Hence, by Proposition 2 the inequality for $\pazocal{V}(x_t(\cdot))$ (if we discard $D(x_t(\cdot))$) becomes
\begin{equation*} \dot{\pazocal{V}}(x_t(\cdot)) +b(t, N_t) \, \pazocal{V}(x_t(\cdot)) \leq b(t, N_t)\pazocal{M}_1(0) e^{-2\int_0^t b(s, N_s) \, \diff s} ,\end{equation*}
and by Gronwall's we get the bound
\begin{align*} \pazocal{V}(x_t(\cdot))\leq \pazocal{V}(x_0(\cdot)) e^{-\int_0^t b(s, N_s) \, \diff s}+\pazocal{M}_1(0) \left(e^{-\int_0^t b(s, N_s) \, \diff s}-e^{-2\int_0^t b(s, N_s) \, \diff s} \right). \end{align*}

A more direct approach is to work with the functional $\pazocal{V}_X(x_t(\cdot))$ introduced in \eqref{eq: Liapunov_X}, since we have seen that it is dissipative when $\dot{X}(t, N_t)=0$. Then, with the help of \eqref{eq: Liapunov_X_Derivative} we get  \begin{equation} \label{eq: X(t)_const} \dot{\pazocal{V}}_{X}(x_t(\cdot))=-b(t, N_t) \, \pazocal{V}_{X}(x_t(\cdot)) -D(x_t(\cdot)) .\end{equation}
Once again, ignoring the term $D(x_t(\cdot))$ and integrating \eqref{eq: X(t)_const} in time we get \eqref{eq: V_Xconst}.

We remark that the structure of $\psi(r)$ plays no role when $X(t, N_t)$ is constant, and as one should expect by the fact that practically all the population enters with a specific opinion. The convergence rate depends on how fast $N_t \to \infty$.

\end{proof}

\begin{proof}[Proof of Theorem \ref{thm: Convergence1}] The proof is contained in the results of Propositions \ref{proposition: clustering}-\ref{proposition: C1_criterion}.
\end{proof}

\begin{proof}[Proof of Theorem \ref{thm: Convergence2}]

(i) Since we have an all-to-all interaction topology then \begin{equation} \label{eq: psi_bound}  \inf \limits_{s,s' \in [0,N_t]}\psi(|x_t(s)-x_t(s')|)= \psi_*> 0, \qquad \forall t \geq 0. \end{equation}
Now Proposition 3 implies that $\int_0^{\infty}\pazocal{V}(x_t(\cdot)) \, dt <\infty$. It remains to show that $\pazocal{V}(x_t(\cdot))$ is uniformly continuous in time. The latter follows from the $L^\infty$ bound of $b(t, N_t)$ which implies that $|\dot{\pazocal{V}}(x_t(\cdot))|$ is uniformly bounded for all $t>0$. With the help of Barbalat's lemma it follows that $\pazocal{V}(x_t(\cdot)) \to 0$ as $t \to \infty$. The convergence in measure follows trivially by Chebyshev's inequality.
\\
(ii) Applying Proposition 2 and using \eqref{eq: psi_bound} we have
\begin{equation*} \dot{\pazocal{V}}(x_t(\cdot))\leq -b(t, N_t)\pazocal{V}(x_t(\cdot))+b(t, N_t)\pazocal{M}_1(t) -\psi_*\pazocal{V}(x_t(\cdot)) . \end{equation*}
If we assume that $b(t,N_t)\xrightarrow{t \to 0} 0$ the main dissipation comes from the term $-\psi_*\pazocal{V}(x_t(\cdot))$ and we can discard $-b(t, N_t)\pazocal{V}(x_t(\cdot))$, so we get
\begin{equation*} \pazocal{V}(x_t(\cdot)) \leq e^{-\psi_*t}\pazocal{V}(x_0(\cdot)) +\int_0^t b(s, N_s) \pazocal{M}_1(s) e^{-\psi_*(t-s)} \, \diff s . \end{equation*}
The result follows from Lemma \ref{lemma: C1'_condition}.
\end{proof}

\section{Macroscopic (Kinetic) formulation of the consensus models}
\label{sec: Macroscopic}
\subsection{Formal derivation of the Macroscopic equation}

We now introduce the empirical measure that measures the probability density of particles at time $t$
\begin{equation} \label{eq: emp_meas} f_t(x):=\frac{1}{N_t} \int_0 ^{N_t} \delta_{x_t(s')}(x) \, \diff s' .\end{equation}
Let's comment a bit on the definition of the empirical measure $f_t(x)$. We know that in the classical kinetic theory treatment, when we deal with the discrete system of $N$ particles \eqref{eq: consensus}, the standard definition of the empirical measure is $\mu^N_t(x)=\frac{1}{N}\sum _{k=1}^{N} \delta_{x_k(t)}(x)$. The only consistent definition that gives a probability measure for the continuous model with growing population is \eqref{eq: emp_meas}.

\begin{proposition} \label{proposition: kinet_derivation}
For every solution $x_t(s)$ of \eqref{eq: conses_symm}, the empirical measure  \eqref{eq: emp_meas} associated with $x_t(s)$ satisfies the weak-formulation of \eqref{eq: kinet}, i.e. for all test function $\varphi(t, x)\in C_c^1([0,T) \times \mathbb{R}^d)$ with compact support on $[0,T) \times \mathbb{R}^d$ we have
\begin{equation} \label{eq: Weak_form} \begin{aligned}\int_0^T \int_{\mathbb{R}^d} \Big( \partial_t \varphi(t, x) +  V[f_t](x) \cdot \nabla \varphi(t, x) &- b(t, N_t) \varphi(t,x) \Big) f_t(x) \, \diff x \, \diff t  = \\ &-\int_0^T  b(t, N_t) \varphi(t , X(t, N_t)) \, \diff t  - \int_{\mathbb{R}^d} \varphi(0,x) f_0(x) \, \diff x . \end{aligned}\end{equation}
\end{proposition}

\begin{proof}
If we multiply $f_t(x)$ with some test functions $\varphi(t, x) \in C_c^1([0,T) , \mathbb{R}^d)$, integrate over $\mathbb{R}^d$ and differentiate in time we get
\begin{equation} \label{eq: Express1} \begin{aligned} & \frac{d}{dt} \int_{\mathbb{R}^d} f_t(x) \varphi(t, x) \, \diff x = \frac{d}{dt} \left( \frac{1}{N_t} \int_0^{N_t} \varphi(t, x_t(s')) \, \diff s' \right) \\ & =-\frac{\dot{N}_t}{N_t^2}\int_0^{N_t} \varphi(t, x_t(s')) \, \diff s' + \frac{\dot{N}_t}{N_t} \varphi(t, X(t, N_t)) +  \frac{1}{N_t} \int_0^{N_t} \frac{d}{dt} \varphi(t, x_t(s')) \, \diff s' \\ & = - \int_{\mathbb{R}^d} b(t, N_t) \varphi(t, x)f_t(x) \, \diff x + \int_{\mathbb{R}^d} b(t, N_t) \varphi(t, x) \delta_{X(t, N_t)}(x) \, \diff x + \frac{1}{N_t} \int_0^{N_t} \frac{d}{dt} \varphi(t, x_t(s')) \, \diff s' \\ &= \int_{\mathbb{R}^d} h[f_t](x) \varphi(t, x) \, \diff x + \frac{1}{N_t} \int_0^{N_t} \frac{d}{dt} \varphi(t, x_t(s')) \, \diff s'. \end{aligned} \end{equation}
The last integral term in \eqref{eq: Express1} is
\begin{equation} \label{eq: Express2} \begin{aligned} & \frac{1}{N_t} \int_0^{N_t} \frac{d}{dt} \varphi(t, x_t(s')) \, \diff s' =\frac{1}{N_t} \int_0^{N_t} \Big( \partial_t \varphi(t, x_t(s'))+\nabla \varphi(t, x_t(s')) \cdot \dot{x}_t(s') \Big) \, \diff s' \\ &= \int_{\mathbb{R}^d} \partial_t \varphi(t,x) f_t(x) \, \diff x + \frac{1}{N_t} \int_0^{N_t} \nabla \varphi(t, x_t(s')) \cdot \, \left( \frac{1}{N_t} \int_0^{N_t} \psi(|x_t(s'')-x_t(s')|)(x_t(s'')-x_t(s')) \, \diff s'' \right) \diff s' \\ & = \int_{\mathbb{R}^d} \partial_t \varphi(t,x) f_t(x) \, \diff x + \frac{1}{N_t}\int_0^{N_t} \nabla \varphi(t, x_t(s'))  \cdot \left( \int_{\mathbb{R}^d} f_t(y) \psi(|y-x_t(s')|) (y-x_t(s')) \, \diff y \right) \diff s' \\ &= \int_{\mathbb{R}^d} \partial_t \varphi(t,x) f_t(x) \, \diff x  + \int_{\mathbb{R}^d} f_t(x)  \nabla \varphi(t, x) \cdot \left( \int_{\mathbb{R}^d} f_t(y) \psi(|y-x|) (y-x) \, \diff y \right) \diff x \\ & =\int_{\mathbb{R}^d} \partial_t \varphi(t,x) f_t(x) \, \diff x +  \int_{\mathbb{R}^d} f_t(x)  V[f_t](x) \cdot \nabla \varphi(t,x) \, \diff x .\end{aligned} \end{equation}
The last step is to integrate \eqref{eq: Express1} from $0$ to $T$ and use \eqref{eq: Express2} to get \eqref{eq: Weak_form}.
\end{proof}

Inspired by the above derivation we may now give a formal definition for a weak solution to \eqref{eq: kinet}
\begin{definition} Assume an initial profile $f_0 \in \pazocal{P}_1(\mathbb{R}^d)$ with compact support and a measurable $X(t, N_t)$ for all $t>0$, and take any $T>0$. We say that a measurable $f_t \in C([0,T);\pazocal{P}_1(\mathbb{R}^d))$ is a weak solution of \eqref{eq: kinet}, iff for any $\varphi(t, x)\in C_c^{\infty}([0,T) \times \mathbb{R}^d)$ with compact support on $[0,T) \times \mathbb{R}^d$ we have that \eqref{eq: Weak_form} holds.
\end{definition}

\begin{proposition} \label{proposition: Kinet_Well-posedness}
Assume an initial non-negative profile $f_0(x)\in \pazocal{P}_1(\mathbb{R}^d)$ with compact support. We also assume a Lipschitz communication $\psi(r)$ and a measurable $X(t, N_t)$ that satisfies \eqref{eq: X(t)_cond}. It follows that there exists a non-negative distributional solution $f_t(x) \in C([0,T];\pazocal{P}_1(\mathbb{R}^d))$ to \eqref{eq: kinet} for any $T>0$. Furthermore, assuming $\mathrm{supp} \, f_0 \subset B(0, \pazocal{R}_{in})$, i.e. the initial profile is supported inside a ball of radius $\pazocal{R}_{in}$ then any solution $f_t(\cdot)$ satisfies \begin{equation*} \mathrm{supp} \, f_t \subset B(0, \pazocal{R}_*), \quad \text{where} \qquad \pazocal{R}_*= \max \{ \pazocal{R}_{in} , X_B \} .\end{equation*}
Here $B(0,R)$ is the standard notation for a ball, centered at $0$, with radius $R$.
\end{proposition}

\begin{proof} The existence part relies heavily on the nice properties of the underlying characteristic flow, and we shall only give a brief sketch of the proof. The main idea follows the methodology discussed in \cite{CaCaRo} for homogeneous alignment models.
We need to construct a sequence $\{ f_t^n \}_{n=1}^{\infty}$ of measurable functions that belong in $C([0,T]; \pazocal{P}_1(\mathbb{R}^d))$ (for some $T>0$) s.t. $f_t^n$ converges to a measurable solution in some appropriate metric space.

We may choose e.g. the iterative scheme
\begin{equation*} \partial_t f_t^{n+1}(x)+\nabla \cdot (V[f_t^n](x)f_t^{n+1}(x))=h[f_t^{n+1}] ,\end{equation*} which can be written differently as
\begin{equation} \label{eq: iteration_f^n} \partial_t f_t^{n+1}(x) +V[f_t^n](x)\cdot \nabla f_t^{n+1}(x)+ \left( b(t, N_t)+\nabla \cdot V[f_t^n](x) \right)f_t^{n+1}=b(t, N_t)\delta_{X(t, N_t)} .\end{equation}
Note that we have the following well-known estimates for $V[f_t](x)$ when $\psi(r)$ is Lipstchitz (see e.g. \cite{MuPe} for a similar calculation) \begin{equation*} |V[f_t](x)-V[f_t](x')| \lesssim |x-x'| , \qquad |V[f_t](x)-V[\widetilde{f}_t](x)| \lesssim W_1(f_t,\widetilde{f}_t) .\end{equation*} Here, when we write $A \lesssim B$, we mean $A \leq C B$ for a constant $C$ that depends only $d$, the Lipschitz constant $L_{\psi}$ of $\psi(\cdot)$, the maximum value $\psi_M$, and the support sets of $f_t$ and $\widetilde{f}_t$. Also $|\cdot|$ is any Euclidean norm in $\mathbb{R}^d$, due to their equivalence property.

Equation \eqref{eq: iteration_f^n} can be solved in $f_t^{n+1}(x)$ (provided we know the previous step $f_t^n(x)$) with the help of the characteristic system
\begin{equation*} \dot{T}^f_t(x)=V[f_t](x), \qquad T^f_0(x)=x .\end{equation*}  We only mention that for a ``fixed'' $f_t$, the solution to this characteristic system has some nice behavior for a Lipschitz regular $\psi(r)$ and provided that the support of $f_t(x)$ is bounded, i.e. we have
\begin{equation*} |T^f_t(x)-T^f_t(x')| \lesssim |x-x'|, \qquad |T^f_t(x)-T^{\widetilde{f}}_t(x)| \lesssim \int_0^t W_1(f_s, \widetilde{f}_s) \, \diff t .\end{equation*}

Going back to the iteration scheme, we can write the semi-explicit solution of $f_t^{n+1}(x)$ in the operator form $f_t^{n+1}=\pazocal{K}(f_t^n)$ where
\begin{equation} \label{eq: iteration} \begin{aligned}\pazocal{K}(f_t^n(x))&=f_0(T^{f^n}_{-t}(x))e^{-\int_0^t (b(s, N_s)+\nabla \cdot V[f_s^n](T^{f^n}_{s-t}(x))) \, \diff s} \\ &+ \int_0^t b(s, N_s) \delta_{X(s, N_s)} e^{-\int_s^t (b(s', N_{s'})+\nabla \cdot V[f_{s'}^n](T^{f^n}_{s'-t}(x))) \, \diff s'}\, \diff s . \end{aligned}\end{equation}
Based on this formulation of $\pazocal{K}(\cdot)$ we can prove the contractivity of the operator with respect to the metric distance \begin{equation*}\pazocal{W}_1(\mu,\nu)=\sup \limits_{t \in [0,T]}W_1(\mu_t,\nu_t)\end{equation*} for small enough $T>0$. The contractivity property implies that $f_t^n(x)$ is Cauchy and hence convergent to some limit $f_t^*(x)$ that is a fixed point of the equation. Note that starting from some $f_0 \geq 0$, non-negativity of each iteration $\{ f^n_t\}^{\infty}_{n=1}$ follows directly from \eqref{eq: iteration}, and hence the non-negativity of its limit. Meanwhile, although uniqueness follows automatically from the contraction property, we give an independent proof of uniqueness below to show how we derive the stability estimate.
\end{proof}

We now turn to the main stability result for proving uniqueness for our kinetic model
\begin{proof} [Proof of theorem \ref{thm: Well_posedness2}]

We start by writing the equation for the difference $(f_t-\widetilde{f}_t)(x)$ of two solutions of \eqref{eq: kinet}
\begin{equation} \label{eq: mu - tildemu} \partial_t (f_t-\widetilde{f}_t)(x) + \nabla \cdot \left(V[f_t](x)f_t(x) -  V[\widetilde{f}_t](x)\widetilde{f}_t(x) \right) +b(t, N_t))(f_t-\widetilde{f}_t)(x)=0 .\end{equation}
Next, we consider the equation
\begin{equation} \label{eq: phi} \partial_t \phi(t,x) +V[f_t](x) \cdot \nabla \phi(t,x) +b(t, N_t) \phi(t,x)=0 , \qquad \phi(T, x)=\bar{\phi}(x) , \end{equation}
which is dual to \eqref{eq: mu - tildemu}. For the sake of the proof we restrict the function $\bar{\phi}$ to Lipschitz functions with constant $1$.
Differentiating \eqref{eq: phi} in the $x$ variable we get
\begin{equation*} \partial_t \nabla \phi(t,x) +V[f_t](x) \cdot \nabla \nabla \phi(t,x) + \nabla V[f_t](x) \cdot \nabla \phi(t,x) + b(t, N_t) \nabla \phi(t,x)=0 ,\end{equation*}
which leads to the following $L^{\infty}$ estimate
\begin{equation*} \frac{d}{dt}\| \nabla \phi(t,\cdot) \|_{L^\infty (\mathbb{R}^d)} \leq  \left(\| b \|_{L^{\infty}([0, T))} + \| \nabla V[f_t](\cdot) \|_{L^{\infty}([0, T)\times \mathbb{R}^d)} \right) \, \| \nabla \phi(t,\cdot) \|_{L^{\infty}(\mathbb{R}^d)} .\end{equation*}
Observe that the boundedness of $\| \nabla V[f_t](\cdot) \|_{L^{\infty}([0, T)\times \mathbb{R}^d)}$ follows by the fact that any solution $f_t$ has bounded support.
The last inequality implies (given $\|\nabla \bar{\phi}\|_{L^{\infty}(\mathbb{R}^d)} \leq 1$) that
\begin{equation*} \|\nabla \phi(t,\cdot)\|_{L^{\infty}(\mathbb{R}^d)} \leq e^{ (T-t) (\| b \|_{L^{\infty}([0, T))}+ \| \nabla V[f_t](\cdot) \|_{L^{\infty}([0, T)\times \mathbb{R}^d)})} \qquad  \text{for} \quad t \in [0, T) .\end{equation*}
Let us set $C_T:=\| b \|_{L^{\infty}([0, T))}+ \| \nabla V[f_t](\cdot) \|_{L^{\infty}([0,T) \times \mathbb{R}^d)}$ to simplify notation. Equations \eqref{eq: mu - tildemu}-\eqref{eq: phi} yield
\begin{equation} \label{eq: mu-phi} \begin{aligned} \partial_t (\phi (t,x) & (f_t -\widetilde{f}_t)(x))+ \phi (t,x) \nabla \cdot \left( V[f_t](x)f_t(x) - V[\widetilde{f}_t](x) \widetilde{f}_t(x) \right) \\ & + \nabla \phi (t,x) \cdot V[f_t](x) (f_t -\widetilde{f}_t)(x) +2 b(t, N_t)\phi (t,x) (f_t -\widetilde{f}_t)(x)=0 .\end{aligned} \end{equation}
Integrating \eqref{eq: mu-phi} in time from $0$ to $T$ and in space over $\mathbb{R}^d$ we get
\begin{align*} \int_{\mathbb{R}^d} \bar{\phi}(x) (f_t-\widetilde{f}_t)(x) \, \diff x &= \int_{\mathbb{R}^d} \phi(0, x)(f_0-\widetilde{f}_0)(x) \, \diff x   -  \int_0^T  \! \! \int _{\mathbb{R}^d} \nabla \phi(t,x) \cdot \left( V[f_t](x)-  V[\widetilde{f}_t](x)\right)f_t (x) \, \diff x \, \diff t \\  & +  2 \int_0^T \! \! \int_{\mathbb{R}^d} b(t, N_t) \phi(t, x) (f_t  - \widetilde{f}_t)(x)  \diff x \, \diff t = \pazocal{I}_1 + \pazocal{I}_2 + \pazocal{I}_3 .\end{align*}
The first term $\pazocal{I}_1$ is bounded by
\begin{equation*} \pazocal{I}_1 \leq e^{C_T  T} W_1(f_0 ,\widetilde{f}_0) . \end{equation*}
For $\pazocal{I}_2$ we can show that $\|V[f_t]-V[\widetilde{f}_t]\|_{L^{\infty}(\mathbb{R}^d)} \leq 2 \pazocal{R}_* \psi_M W_1(f_t ,\widetilde{f}_t)$. Indeed, this follows directly due to
\begin{equation*} V[f_t](x)-V[\widetilde{f}_t](x) =\int_{\mathbb{R}^d} \psi(|y-x|) (y-x) (f_t(y)-\widetilde{f}_t(y) ) \, \diff y .\end{equation*}
As a result \begin{equation*} \pazocal{I}_2 \leq 2 \pazocal{R}_* \psi_M \int_0^T e^{C_T (T-t)} W_1(f_t, \widetilde{f}_t) \diff t .\end{equation*}
Meanwhile the $\pazocal{I}_3$ term is bounded by
\begin{equation*} \pazocal{I}_3 \leq 2 \| b \|_{L^{\infty}([0,T))} \int_0 ^T e^{ C_T (T-t)} W_1(f_t, \widetilde{f}_t) \diff t .\end{equation*}
Putting everything together, and taking the supremum over all $\bar{\phi}$ functions we obtain inequality
\begin{equation*} W_1(f_t, \widetilde{f}_t) \leq e^{C_T T} W_1(f_0 ,\widetilde{f}_0) +(2 \pazocal{R}_* \psi_M + 2 \| b \|_{L^{\infty}([0, T))}) \int_0 ^T e^{ C_T (T-t)} W_1(f_t, \widetilde{f}_t) \diff t .\end{equation*}
The desired estimate follows from  Gr\"onwall's integral inequality.
\end{proof}

\begin{comment}
\begin{remark} The same exact procedure followed for the symmetric system \eqref{eq:conses_symm}, can be used to obtain the kinetic formulation to \eqref{eq:H-K}. After computations, we obtain the equation \eqref{eq:kinet} but with an interaction Kernel $V[f_t]$ given by
\begin{equation} \label{eq:inter_kernel_H-K} V[f_t](x):= \frac{1}{\int \psi(|y-x|) f_t(y) \, dy } \int \psi(|y-x|) (y-x) f_t(y) \, dy . \end{equation}
\end{remark}
\end{comment}

\subsection{Moments and Convergence in the kinetic model}

The corresponding definition of moments for the kinetic model for a solution $f_t(x)$ are
\begin{equation*}  M_0(t)=\int_{\mathbb{R}^d} f_t(x) \, \diff x , \quad M_1(t)=\int_{\mathbb{R}^d} x f_t(x) \, \diff x , \quad  M_2(t)=\int_{\mathbb{R}^d} |x|^2 f_t(x) \, \diff x.\end{equation*}
In a similar fashion like when working with the continuous model we obtain the following result for the evolution of the kinetic moments

\begin{proposition} Let $f_t(x)$ be a weak solution to \eqref{eq: kinet}, with initial data $f_0(x)$. Then, for $t \geq 0$ the following expressions hold for the moments of $f_t(x)$,
\begin{equation} \label{eq: M_iDeriv}
\begin{aligned}
(i) \quad \dot{M}_0(t)&=0,  \\
(ii) \quad \dot{M}_1(t)&=-b(t,N_t)M_1(t)+b(t,N_t)X(t, N_t) , \\
(iii) \quad \dot{M}_2(t)&=-b(t,N_t) M_2(t)+b(t,N_t)|X(t, N_t)|^2-D(f_t(\cdot)) ,
\end{aligned}
\end{equation}
where the dissipation term is now given by
\begin{equation*} D(f_t(\cdot)):=\iint_{\mathbb{R}^{2d}} \psi(|x-y|)|x-y|^2 f_t(x) f_t(y) \, \diff x \, \diff y .\end{equation*}
Hence \eqref{eq: kinet} propagates finiteness of 0, 1st and 2nd moments, i.e. if $f_0 \in \pazocal{P}_i(\mathbb{R}^d)$ then $f_t \in \pazocal{P}_i(\mathbb{R}^d)$ for $i=0,1,2$, for all $t \geq 0$.
\end{proposition}

\begin{proof}
The moment equations are derived by multiplying \eqref{eq: kinet} with the expressions $\{ 1, x, |x|^2 \}$ and integrating over $\mathbb{R}^d$. First we multiply with the transport term $\nabla \cdot (V[f_t](x) f_t(x))$ and after we integrate by parts we get
\begin{equation*} \int_{\mathbb{R}^d}
\left( \begin{array}{c} 1 \\ x \\ |x|^2 \end{array} \right)\nabla \cdot (V[f_t](x) f_t(x))\, \diff x=\left( \begin{array}{c} 0 \\ 0 \\  D(f_t(\cdot)) \end{array} \right) . \end{equation*}
Next step we multiply with the source term $h[f_t](x)$ and integrate to get
\begin{equation*} \int_{\mathbb{R}^d} \left( \begin{array}{c} 1 \\ x \\ |x|^2 \end{array} \right) h[f_t](x)\, \diff x=\left( \begin{array}{c} 0 \\ b(t,N_t) (X(t, N_t)-M_1(t)) \\ b(t,N_t) (|X(t, N_t)|^2-M_2(t)) \end{array} \right)  .\end{equation*}
Putting everything together yields \eqref{eq: M_iDeriv}.
\end{proof}
The similarity between the moment equations in the kinetic and the microscopic description yields the following
\begin{corollary} \label{corollary: M_1}
Let $f_t(\cdot) \in \pazocal{P}_1(\mathbb{R}^d)$  be a solution to the symmetric consensus model \eqref{eq: kinet} with boundary data $X(t, N_t)$. When $N_\infty=\infty$ we have
\begin{equation*}  M_1(t) \xrightarrow{ t \to \infty } \frac{1}{N_t}\int_0^t X(s, N_s) \dot{N}_s\, \diff s  ,\end{equation*}
while for $N_\infty < \infty$ then
\begin{equation*}  M_1(t) \xrightarrow{ t \to \infty } \frac{N_0}{N_\infty}M_1(0)+\frac{1}{N_{\infty}}\int_0^{\infty} X(s, N_s) \dot{N}_s \, \diff s. \end{equation*}
\end{corollary}
For any solution $f_t \in \pazocal{P}_2(\mathbb{R}^d)$ it follows after a simple use of the triangle inequality that $D(f_t(\cdot))<\infty$, for all $t \geq 0$. We also have an estimate for $V[f_t](x)$ from a simple application of the Cauchy-Schwartz inequality
\begin{align*} \label{eq: V_C-S} |V[f_t](x)|^2 &=\Big| \int \psi(|x-y|) (y-x) f_t(y)\, \diff y \Big|^2 \leq \left( \int \psi(|x-y|) |x-y| f_t(y)\, \diff y \right)^2 \\ & \leq \psi_M \int  \psi(|x-y|) |x-y|^2 f_y(y) \, \diff y , \qquad x \in \mathbb{R}^d .\end{align*} This implies that when we
work in the space of densities with bounded second moment, then $V[f_t](x) \in L^2(\mathbb{R}^d ; f_t(x) \, \diff x)$ due to the boundedness of $D(f_t(\cdot))<\infty$.

To study of asymptotic behavior we need to define the macroscopic variance functional and show its evolution in time.
\begin{proposition} Assume a solution $f_t(x) \in \pazocal{P}_2(\mathbb{R}^d)$ of problem \eqref{eq: kinet} for a profile of incoming opinions $X(t, N_t)$. We define the \textbf{macroscopic variance} functional $\pazocal{V}(f_t(\cdot))$ by
\begin{equation} \label{eq: Liapunov_mu} \pazocal{V}(f_t(\cdot)):=\int_{\mathbb{R}^d} |x-M_1(t)|^2 f_t(x)\, \diff x . \end{equation}
The derivative of $\pazocal{V}(f_t(\cdot))$ along this solution satisfies \begin{equation} \label{eq: Liapunov_mu_Der} \dot{\pazocal{V}}(f_t(\cdot))=-b(t, N_t) \pazocal{V}(f_t(\cdot)) +b(t, N_t)| M_1(t)-X(t, N_t)|^2 -D(f_t(\cdot)). \end{equation} \end{proposition}
\begin{proof}
Differentiating \eqref{eq: Liapunov_mu} along solutions of \eqref{eq: kinet} we get
\begin{align*} \dot{\pazocal{V}}(f_t(\cdot))&=-2 \int_{\mathbb{R}^d} \dot{M}_1(t) \cdot \left(x-M_1(t)\right)f_t(x) \, \diff x + \int_{\mathbb{R}^d} |x-M_1(t)|^2 \, \partial_t f_t(x)\, \diff x \\ &= \int_{\mathbb{R}^d} |x-M_1(t)|^2 \left( -\nabla \cdot(V[f_t](x)f_t(x))+h[f_t](x) \right) \, \diff x \\ &=2 \int_{\mathbb{R}^d} (x-M_1(t)) \cdot V[f_t](x)f_t(x)\, \diff x +\int_{\mathbb{R}^d} |x-M_1(t)|^2 h[f_t](x)\, \diff x = \pazocal{T}_1 + \pazocal{T}_2 . \end{align*}
The first term $\pazocal{T}_1$ equals the dissipation $-D(f_t(\cdot))$, i.e.
\begin{align*} \pazocal{T}_1  =  2\iint_{\mathbb{R}^{2d}} \psi(|x-y|)  \, x \cdot (y-x)f_t(x) f_t(y)\, \diff x \, \diff y  =  -D(f_t(\cdot)), \end{align*}
while term $\pazocal{T}_2$ is
\begin{align*} \pazocal{T}_2=\int_{\mathbb{R}^d} |x-M_1(t)|^2 h[f_t](x)\, dx=-b(t, N_t) \pazocal{V}(f_t(\cdot)) + b(t, N_t)|X(t, N_t)-M_1(t)|^2 . \end{align*}
Combining the two terms $\pazocal{T}_1$ and $\pazocal{T}_2$ yields \eqref{eq: Liapunov_mu_Der}.
\end{proof}

Next, we give the main estimate for the long time asymptotics for when $N_{\infty}< \infty$.
\begin{lemma} \label{lemma: D(f)_regularity} Let $N_{\infty} < \infty $. Assume also that the interaction kernel is Lipschitz, i.e. $\psi(\cdot) \in W^{1,\infty}(\mathbb{R})$.
For any solution $f_t \in C([0,\infty); \pazocal{P}_1(\mathbb{R}^d))$ of \eqref{eq: kinet} with compactly supported initial data
$f_0(x) \in \pazocal{P}_2(\mathbb{R}^d)$ we have that $D(f_t(\cdot)) \xrightarrow{ t \to \infty } 0$. Moreover, if we have that $f_t \rightharpoonup^*  f_{\infty}$, as $t \to \infty$, for
some steady state $f_{\infty} \in \pazocal{P}_1(\mathbb{R}^d)$ then $D(f_{\infty}(\cdot))=0$.
\end{lemma}

\begin{proof} We start with the observation that \eqref{eq: Liapunov_mu_Der} implies that $\int_0^{\infty} D(f_t(\cdot))\, \diff t <\infty$ (since $\int_0^{\infty} b(t, N_t)\, \diff t < \infty$). We need to show that $D(f_t(\cdot))$ cannot fluctuate wildly as a function of time, i.e. that $\dot{D}(f_t(\cdot))$ is bounded uniformly in time.

Differentiating $D(f_t(\cdot))$ in time we get
\begin{equation} \label{eq: dotD(t)} \begin{aligned} \dot{D}(f_t(\cdot))=& 2 \iint_{\mathbb{R}^{2d}} \psi(|x-y|)|x-y|^2 \partial_t f_t(x) f_t(y) \, \diff x \diff y \\
=& 2 \iint_{\mathbb{R}^{2d}} \left(\psi'(|x-y|)|x-y| + 2 \psi(|x-y|) \right) V[f_t](x) \cdot (x-y)  f_t(x) f_t(y) \, \diff x \diff y  \\
+& 2 \int_{\mathbb{R}^{d}} \psi(|X(t, N_t)-y|)|X(t, N_t)-y|^2 b(t, N_t) f_t(y) \, \diff y \\
-& 2 b(t, N_t) \iint_{\mathbb{R}^{2d}} \psi(|x-y|)|x-y|^2 f_t(x) f_t(y) \, \diff x \diff y =2\pazocal{I}_1(t)+2\pazocal{I}_2(t) -2 b(t, N_t) D(f_t(\cdot)). \end{aligned} \end{equation}
Terms $\pazocal{I}_1(t)$ and $\pazocal{I}_2(t)$ in \eqref{eq: dotD(t)} are easily bounded by
\begin{equation*} |\pazocal{I}_1(t)| \leq 2 \pazocal{R}_* d (2 \pazocal{R}_* \|\psi'\|_{L^{\infty}(\mathbb{R}_+)} + 2 \psi_M ) \|V[f_t](\cdot)\|_{L^{\infty}([0,+\infty) \times \mathbb{R}^d)} ,\end{equation*}
and
\begin{equation*} |\pazocal{I}_2(t)| \leq \psi_M \|b \|_{L^{\infty}([0,+\infty))} (X_B + \pazocal{R}_*)^2 .\end{equation*}
Hence, $\dot{D}(f_t(\cdot))$ is bounded.
\end{proof}

The following easy corollary follows when $\psi(r)>0$,

\begin{corollary} \label{corollary: psi_pos}
Let $N_{\infty} < \infty $. Assume also that the interaction kernel $\psi(r)$ is Lipschitz and positive.
For any solution $f_t \in C([0,\infty); \pazocal{P}_1(\mathbb{R}^d))$ of \eqref{eq: kinet} with compactly supported initial data
$f_0(x) \in \pazocal{P}_2(\mathbb{R}^d)$ we have that $\pazocal{V}(f_t(\cdot)) \xrightarrow{ t \to \infty } 0$, i.e. the solution concentrates around $M_1(t)$.
\end{corollary}

\begin{proof} This follows directly from the fact that the support of any solution is bounded so we have that $|x-y|\leq 2 \pazocal{R}_*$ for all $x, y \in supp f_t$, $\forall t \geq 0$. The continuity of $\psi(r)$ implies that
\begin{equation} \label{eq: psi_*}\inf \limits_{x,y \, \in \, \mathrm{supp} f_t, \, \forall t\geq 0} \psi(|x-y|)=\psi_* > 0. \end{equation} This in turn implies that $\iint_{\mathbb{R}^{2d}} |x-y|^2 f_t(x) f_t(y) \, \diff x \diff y \xrightarrow{ t \to \infty } 0$ or $\pazocal{V}(f_t(\cdot)) \xrightarrow{ t \to \infty } 0$.
\end{proof}

The following lemma indicates that when the c.w. is of Type II then the solution concentrates into Dirac masses that are mutually distanced by a distance $\geq 1$.
\begin{lemma} \label{lemma: Concentration}[Concentration of mass] Let $N_{\infty}<\infty$ and consider a c.w. $\psi(r)$ that is Lipschitz continuous with compact support on $[0,1]$. Take any $\epsilon>0$, then for any two disjoint sets $A, B \subset \mathbb{R}^d$ that lie in a ball of radius $1$, i.e. $B(0,1)$, there exists no sequence of times $t_n \to \infty$ s.t. $ \min \{ m_n(A), m_n(B) \} > \epsilon$, where $m_n(A)=\int_A f_{t_n}(x) \, \diff x > \epsilon$ and similarly $m_n(B)=\int_B f_{t_n}(x) \, \diff x > \epsilon$. Furthermore, if $D(f_{\infty}(\cdot))=0$, then $f_{\infty}(x)= \sum \limits_{i=1}^J a_i \delta_{x_i}(x)$ for some $J \in \mathbb{N}_+$, and $a_i$, $x_i$ having the properties given in Theorem \ref{thm: Convergence_kin}.
\end{lemma}

\begin{proof}
We begin by assuming that such a sequence of times $t_n \xrightarrow{ n \to \infty } \infty$ exists. We use the definition of $D(f_t(\cdot))$ at times $t_n$, but we observe that since the sets $A , B$ are compact and disjoint we have that $\inf \limits_{(x,y)\in A \times B}\psi(|x-y|)|x-y|^2 >0$. Thus, if we restrict $D(f_t(\cdot))$ over the set $A \times B$ we get the inequality
\begin{equation*} \epsilon^2 < \iint_{A \times B} f_{t_n}(x) f_{t_n}(y) \, \diff x \diff y \leq C_{A,B} D(f_{t_n}(\cdot)) ,\end{equation*}
where $C_{A,B}=\sup \limits_{(x,y)\in A \times B}\frac{1}{\psi(|x-y|)|x-y|^2}$. This inequality cannot hold if $D(f_{t_n}(\cdot)) \xrightarrow{ n \to \infty } 0$, and such a sequence $t_n \xrightarrow{ n \to \infty } \infty$ cannot exist. The proof of the second part of the statement is given in \cite{CaFaTi}(see pg. $9-10$).
\end{proof}

We are now in position to prove Theorem \ref{thm: Convergence_kin}
\begin{proof}[Proof of Theorem \ref{thm: Convergence_kin}]

\textbf{Case $N_{\infty} < \infty$ :}

We focus in the case of a $\psi(r)$ with a compact support, as the case of $\psi(r)>0$ has been answered in Corollary \ref{corollary: psi_pos}.

First we need to show that for every test function $\phi \in C_c^{\infty}(\mathbb{R}^d)$ with compact support, the limit $\lim \limits_{t \to \infty} \int \phi(x) f_t(x) \, \diff x$ exists and is finite. To prove this we first set $F(t)=\int \phi(x) f_t(x) \, \diff x$ and we show that $\int_0^{\infty} |\dot{F}(t)| \, \diff t <\infty$. Indeed, if we multiply \eqref{eq: kinet} by $\phi(x)$ and integrate we get
\begin{equation} \label{eq: F(t)} \dot{F}(t)+b(t, N_t) F(t) = b(t, N_t) \phi(X(t, N_t)) +\int_{\mathbb{R}^d} \nabla \phi(x) \cdot V[f_t](x)f_t(x) \, \diff x . \end{equation}
Note that we can use a standard symmetrization argument to rewrite the last term as
\begin{align*} \int_{\mathbb{R}^d} \nabla \phi(x) \cdot V[f_t](x)f_t(x) \, \diff x &= \int_{\mathbb{R}^{d}}  \int_{\mathbb{R}^{d}} \psi(|x-y|) \nabla \phi(x) \cdot (y-x) f_t(y) f_t(x) \, \diff x \, \diff y
\\ &=\frac{1}{2}\int_{\mathbb{R}^{d}} \int_{\mathbb{R}^{d}} \psi(|x-y|) (\nabla \phi(x)-\nabla \phi(y)) \cdot (y-x) f_t(y) f_t(x) \, \diff x \, \diff y.\end{align*}
Now the Lipschitz regularity of $\nabla \phi(x)$ implies that this term can be controlled by
\begin{equation*}  \Big | \int_{\mathbb{R}^d} \nabla \phi(x) \cdot V[f_t](x)f_t(x) \, \diff x \Big | \leq K D(f_t(\cdot)), \qquad \text{for some} \quad K>0 . \end{equation*}
We can solve \eqref{eq: F(t)} to get the semi-explicit solution
\begin{align*} F(t)=e^{-\int_0^t b(s, N_s)\, \diff s}F(0) &+ \int_0^t e^{-\int_s^t b(s', N_{s'})\, \diff s'} b(s, N_s) \phi(X(s, N_s)) \diff s \\ &+ \int_0^t \int_{\mathbb{R}^d} e^{-\int_s^t b(s', N_{s'})\, \diff s'} \nabla \phi(x) \cdot V[f_s](x)f_s(x) \, \diff x \, \diff s .\end{align*}
Hence since $\int_0^{\infty} D(f_t(\cdot))\, \diff t <\infty$ we easily see that $F(t)$ is uniformly bounded in time, i.e. $|F(t)|\leq C, \quad \forall t\geq 0$. Going back to \eqref{eq: F(t)} and using the uniform bound of $F(t)$ and the fact that $N_{\infty} <\infty$ we get that $\int |\dot{F}(t)| \, \diff t <\infty$. This proves that $f_t$ converges to some measure $f_{\infty}$ in the distributional sense. Furthermore, with the help of fact that $f_t$ has bounded second moments, it can be shown that $f_{\infty}$ is also a probability measure (see \cite{CaFaTi} for more details). Finally, a direct consequence of Lemmas \ref{lemma: D(f)_regularity} and \ref{lemma: Concentration} yields that $f_{\infty}(x)$ is a weighted sum of Dirac masses with the properties given in Theorem \ref{thm: Convergence_kin}.

\textbf{Case $N_{\infty}=\infty$ :}
In the case where $\psi(r)>0$ we have due to \eqref{eq: psi_*}
\begin{equation*} \pazocal{V}(f_t(\cdot)) \leq e^{-2\psi_* t} \pazocal{V}(f_0(\cdot)) + \int_0^t e^{-2\psi_* (t-s)} b(s, N_s)|M_1(s)-X(s, N_s)|^2 \, \diff s .\end{equation*}
Thus, if $b(t, N_t) \xrightarrow{ t \to \infty } 0$, then $\pazocal{V}(f_t(\cdot)) \xrightarrow{ t \to \infty } 0$.

If we know that $(C1)$ condition holds, then from $\dot{\pazocal{V}}(f_t(\cdot)) \leq -b(t, N_t) \pazocal{V}(f_t(\cdot)) + b(t, N_t) |M_1(t)-X(t, N_t)|^2$ we get directly that
\begin{equation*} \pazocal{V}(f_t(\cdot)) \leq e^{-\int_0^t b(s, N_s) \, \diff s} \pazocal{V}(f_0(\cdot)) +\int_0^t \frac{d}{ds}\left( e^{-\int_s^t b(s', N_{s'}) \, \diff s'} \right) |M_1(s)-X(s, N_s)|^2 \, \diff s .\end{equation*}
So with the help of Lemma \ref{lemma: C1'_condition} we get that $\pazocal{V}(f_t(\cdot)) \xrightarrow{ t \to \infty } 0$.

\end{proof}

\section{Conclusions and Open Problems}
\label{sec: Conclusions}
Population growth and the role that it plays in the collective behavior of IBMs has been systematically ignored from the enormous literature devoted to self organized systems. In this paper, we introduced and studied a continuous opinion model with a population of agents that increases according to some general growth rule characterized by a growth rate $b(t,N_t)$. We studied how the choice of rate  $b(t,N_t)$, as well as the opinions of the added population profile $X(t, N_t)$ determine the long time dynamics and asymptotics of our system.
There are many open problems and extensions of the theory we developed in this paper. We list some of them here:
\\
\textbf{Stochasticity of the opinions of incoming agents.} In this work we developed a theory where the boundary condition $X(t, N_t)$ that carries all the information for the opinions of incoming individuals is of the deterministic type. In reality, it may be more appropriate to consider that the opinions of incoming agents showcase some randomness and thus consider $X(t, N_t)$ as some stochastic process to capture this random nature for newly added states. A very basic choice is to consider e.g. \begin{equation} \label{eq: X_stoch} X(t, N_t) \sim f(t, N_t) W_t ,\end{equation}
where $f(t, N_t)$ is some unknown function, and $W_t$ is the standard Brownian motion with increments $W_t-W_s$ that follows the normal distribution $\pazocal{N}(0,t-s)$ ($0 \leq s \leq t$). In this new setting, the unknown function $x_t(s)$ that solves \eqref{eq: conses_symm} is not any more deterministic but stochastic in nature. It is therefore very important to define a proper notion of solution to \eqref{eq: conses_symm} so that the problem remains well-posed in some sense. Another very important question to consider due to the stochasticity involved in \eqref{eq: X_stoch} is the appropriate notion that we must use in order to study the convergence properties of a solution $x_t(s)$.
\\
\textbf{Hegselmann-Krause model.} The original non-symmetric H-K model takes the following form in the continuous agent setting with a growing number of agents, i.e.
\begin{equation} \label{eq: H-K} \dot{x}_t(s)=\int_0 ^{N_t} \frac{\psi(|x_t(s')-x_t(s)|)}{\int_0 ^{N_t} \psi(|x_t(s'')-x_t(s)|) \, \diff s''} (x_t(s')-x_t(s))\, \diff s', \quad 0 \leq s < N_t . \end{equation} Well-posedeness of \eqref{eq: H-K} and the way that the growth rate $b(t,N_t)$ and the boundary condition $X(t, N_t)$ affect consensus reaching and creation of clusters are natural questions. Meanwhile, a very good modeling assumption that can be used in both the symmetric and non-symmetric Hegselmann-Krause is to consider a non-deterministic boundary condition $X(t, N_t)$.
\\
\textbf{Higher-Order models.} Another interesting direction is the treatment of higher-order models that use some sort of \textbf{environmental averaging} \cite{Shv}, as opposed to the first-order opinion model in this paper. Higher order models are more suited for the study of flocks and swarms where the unknown variables are positions and velocities as opposed to opinion models where the unknown states are modeled as positions in $\mathbb{R}^d$.  Environmental averaging is a term coined for systems that evolve following some sort of averaging rule in the space of velocities or in other mechanical variables. The most prominent example of a flocking model that follows an environmental averaging rule is the now classical Cucker-Smale model \cite{CuSm}. Other higher-order models may include not only position and velocity variables, but also spins of agents such as the Inertial Spin model which is better suited for the study of turning in flocks \cite{CaCaGiGr, HaKiKiShi, Mar1}. Such models can also be studied under the effect of population growth in the number of interacting agents.
\\
\textbf{More realistic discrete population models that incorporate both birth and death of individuals.} The consensus model that we discussed in this article serves only as a vehicle to understand how the two components of birth rate and the opinions of incoming population both play a major role in understanding their influence in collective behavior. That being said, there are some obvious drawbacks in the proposed model, with the most obvious one being the fact that we have not yet accounted for any population exiting the system, i.e. death of individuals. We have deliberately omitted any instance of individuals exiting the system because death introduces extra modeling difficulties.

To account for both birth and death it is more appropriate to consider a discrete system where the total population is stochastic in nature both for the particles that enter and leave. More specifically, the challenge with death is that we need to find a proper way to pick the population that we remove, since each individual has a specific state as they exit, and the time that a particle spends interacting with others is of crucial importance. A realistic modeling assumption is a scenario where individuals exit with a probability that depends on the ``age'' i.e. the time that each agent is in the system. Ideally, if we draw inspiration from living systems we know that the age of each individual is a good determining factor of the chance of survival in the next fixed time interval, as is manifested in many Actuarial Life tables. Thus, a good starting point is to assume that each agent has a life expectancy that follows a distribution that is consistent with certain actuarial tables.

\section{Appendix} \label{sec: Appendix}
\subsection{Proof of Lemmas \ref{lemma: Bound}--\ref{lemma: C1'_condition}}

\begin{proof}[Proof of Lemma \ref{lemma: Bound}]

Assume that $|x_{t^*}(s^*)|=\pazocal{R}$ at some time $t^*>0$ and for some $s^* \in [0,N_{t^*}]$. Assume also that $t^*$ is the first time that an opinion has magnitude $\pazocal{R}$, and for all $s \in [0,N_{t^*}]$ we have that $|x_{t^*}(s)| \leq \pazocal{R}$. Let also $B(0,\pazocal{R})$ be the d-dimensional ball with center $0$ and radius $\pazocal{R}$, and consider also the tangent plane to $B(0,\pazocal{R})$ at the point $x_{t^*}(s^*)$. We show that $\dot{x}_{t^*}(s^*)$ points in the direction of the tangent plane that includes the ball, and thus the magnitude of $x_{t^*}(s^*)$ cannot exceed $\pazocal{R}$. Indeed, we have that
\begin{equation*} \dot{x}_t(s^*)\Big|_{t=t^*}=\frac{1}{N_{t^*}}\int_0^{N_{t^*}} \psi(|x_{t^*}(s')-x_{t^*}(s^*)|)(x_{t^*}(s')-x_{t^*}(s^*)) \, \diff s'.\end{equation*} Since we can always normalize the interaction Kernel, we may assume that \begin{equation*}\frac{1}{N_{t^*}}\int_0^{N_{t^*}} \psi(|x_{t^*}(s')-x_{t^*}(s^*)|) \, \diff s' =1 .\end{equation*} Then we rewrite the above as \begin{equation*}  \dot{x}_t(s^*)\Big|_{t=t^*}=\frac{1}{N_{t^*}}\int_0^{N_{t^*}} \psi(|x_{t^*}(s')-x_{t^*}(s^*)|) x_{t^*}(s') \, \diff s' -  x_{t^*}(s^*) .\end{equation*} The first term $\frac{1}{N_{t^*}}\int_0^{N_{t^*}} \psi(|x_{t^*}(s')-x_{t^*}(s^*)|) x_{t^*}(s') \, \diff s'$ is a (continuous) convex combination of vectors inside $B(0,\pazocal{R})$. This means that in effect it is also a vector in $B(0,\pazocal{R})$. Since $x_{t^*}(s^*)$ is a vector on the surface of $B(0,\pazocal{R})$, we have that $ \dot{x}_t(s^*)\Big|_{t=t^*}$ as a vector points in the part of the space $\mathbb{R}^d$ (separated by the tangent plane) that includes the ball $B(0,\pazocal{R})$. This implies that $x_{t^*}(s^*)$ cannot escape $B(0,\pazocal{R})$. \end{proof}

\begin{proof}[Proof of Lemma \ref{lemma: Differ}]

We begin by writing the difference of the integral from $t$ to $t+h$ for some $h \ll 1$
\begin{align*} \int_0^{N_{t+h}} G(t+h,x_{t+h}(s))\, \diff s -\int_0^{N_{t}} G(t,x_{t}(s)\, \diff s &= \int_{N_t}^{N_{t+h}} G(t+h,x_{t+h}(s))\, \diff s \\ &+ \int_0^{N_t}\left( G(t+h,x_{t+h}(s)) - G(t,x_{t}(s) \right)\, \diff s. \end{align*}

The last integral term can be computed to its leading order by a simple Taylor expansion, i.e.
\begin{equation*} \int_0^{N_t} \left( G(t+h,x_{t+h}(s)) - G(t,x_{t}(s) \right) \, \diff s =h \int_0^{N_t}\left(\partial_t G(t,x_t(s))+G_x(t,x_t(s)) \cdot \dot{x}_t(s) \right) \, \diff s +O(h^2) . \end{equation*}
Since we have continuity w.r.t $s$ of $x_t(s)$ around $s=N_t$, we may use the mean value to write
\begin{align*} \int_{N_t}^{N_{t+h}} G(t+h,x_{t+h}(s)) \, \diff s &= (N_{t+h}-N_t) \, G(t+h,x_{t+h}(s'))\\&=h\dot{N}_t \, G(t+h,x_{t+h}(s')) +O(h^2) \quad \text{for some} \quad s' \in (N_t, N_{t+h}). \end{align*}

\end{proof}

\begin{proof}[Proof of Lemma \ref{lemma: C1'_condition}]

We fix $\delta>0$, and since $g(t) \xrightarrow{ t \to \infty } 0$ we know there exists some $T_{\delta}>0$ s.t. $|g(t)| < \delta$ for $t > T_{\delta}$. We now decompose integral as
\begin{align*} \frac{1}{N_t}\int_0^t g(s) \dot{N}_s  \, \diff s =\frac{1}{N_t}\int_0^{T_{\delta}} g(s) \dot{N}_s \, \diff s + \frac{1}{N_t} \int_{T_{\delta}}^t g(s) \dot{N}_s \, \diff s . \end{align*}
It is important to note that both integral terms can be controlled for large values of $t$, i.e.
\begin{equation*} \begin{aligned} \bigg|\frac{1}{N_t} \int_0^{T_{\delta}} g(s)  \dot{N}_s \, \diff s \bigg| & \leq \frac{1}{N_t} \|g  \|_{L^\infty([0,\infty))} \left( N_{T_{\delta}}-N_0 \right) .\end{aligned} \end{equation*}
Also,
\begin{equation*} \Bigg| \frac{1}{N_t}\int_{T_{\delta}}^t g(s) \dot{N}_s \, \diff s \Bigg| \leq \delta \left( 1 - \frac{N_{T_{\delta}}}{N_t}\right) < \delta . \end{equation*}
Now, since we have $N_\infty=\infty$ we may choose $T'_{\delta}>T_{\delta}>0$ s.t. for all $t > T'_{\delta}$ we have $ \|g \|_{L^\infty([0,\infty))} \frac{N_{T_{\delta}}}{N_t}<\delta$. This implies that
\begin{equation*} \Bigg| \frac{1}{N_t}\int_0^t g(s) \dot{N}_s \, \diff s \Bigg| <2 \delta , \quad \text{for any}\quad t>T'_{\delta}. \end{equation*} The choice of $\delta$ was originally arbitrary and the result follows.
\end{proof}

\begin{comment}
\section{Continuous Second Order alignment model}

A continuous version for a C-S type of second order alignment model is

\begin{align*} \dot{v}(t,s) &= \frac{1}{N(t)} \int_0^{N(t)} \psi(|x(t,s')-x(t,s)|)(v(t,s')-v(t,s)) \, \diff s' , \\ \dot{x}(t,s) &= v(t,s) ,\qquad  0 \leq s \leq N(t), \quad t \geq 0 . \end{align*}
Here the two state variables are the position $x(t,s)$ and velocity $v(t,s)$ vectors in $\mathbb{R}^d$, computed at time $t$, for an index value $0 \leq s \leq N(t)$.

The main object of the kinetic formulation is to study the probability density
\begin{equation*} f(t,x,v)=\frac{1}{N(t)} \int_0^{N(t)} \delta(v(t,s')-v)\delta(x(t,s')-x) \, \diff s' , \end{equation*}
which measures the density of state variables located at $(x,v)$ at time $t$.

The kinetic equation for the continuous C-S alignment model with a growing population is

\begin{equation} \partial_t f(t,x,v)+v \cdot \nabla_xf(t,x,v)+ \nabla_v \cdot(f(t,x,v)L[f](t,x,v))=h[f](t,x,v), \end{equation}
where the operator $L[f]$ is defined by

\begin{equation*} L[f](t,x,v)=\int \psi(|x-y|) (v_*-v) f(t,y,v_*)\, dv_* dy , \end{equation*}
and the source term $h[f]$ is
\begin{equation} \label{eq:source_CS} h[f](t,x,v) = c(t) \left( \delta(X_t-x) \delta(V_t-v)-f(t,x,v) \right), \quad X_t=x(t,N(t)), \quad V_t=v(t,N(t)) .\end{equation}

\end{comment}

\end{document}